\documentclass[11pt]{amsart}

\usepackage[utf8x]{inputenc} 
\usepackage[a4paper]{geometry} 
\usepackage{enumerate} 
\usepackage{amsmath} 
\usepackage{amssymb}
\usepackage{color,colortbl}
\usepackage{xcolor}
\usepackage[most]{tcolorbox}
\usepackage{fancyvrb}   
 
\definecolor{airforceblue}{rgb}{0.2, 0.2, 0.6}
\definecolor{bleudefrance}{rgb}{0.0, 0.0, 1.0}
\definecolor{darkorchid}{rgb}{0.6, 0.2, 0.8}
\definecolor{darkorange}{rgb}{1.0, 0.55, 0.0}
\definecolor{darkspringgreen}{rgb}{0.09, 0.45, 0.27}
\definecolor{commentoutput}{rgb}{0.33,0.38,0.44}
\definecolor{output}{rgb}{0.24, 0.08, 0.08}
\definecolor{circOut}{rgb}{0.0, 0.81, 0.82}
\definecolor{Gray}{gray}{0.9}  

\usepackage[arrow,curve,matrix,tips,frame]{xy}
\usepackage{adjustbox}
\usepackage{multirow}
\usepackage{url}
\usepackage{hyperref}  
                 \hypersetup{ pdfborder={0 0 0}, 
                              colorlinks=true, 
                              linkcolor=blue, 
                              urlcolor=blue,
                              citecolor=blue,
                              linktoc=page,
                              pdfauthor={Giovanni Staglian{\`o}}, 
                              pdftitle={Explicit computations with cubic fourfolds and GM fourfolds}
                              } 
\usepackage{verbatim}

\theoremstyle{plain} 
\newtheorem{proposition}{Proposition}[section] 
\newtheorem{theorem}[proposition]{Theorem}

\newtheorem{conjecture}[proposition]{Conjecture} 
\theoremstyle{definition}

\newtheorem{codeexample}[proposition]{Code Example} 
\theoremstyle{remark} 
\newtheorem{remark}[proposition]{Remark} 
                                    
\newcommand{\PP}{{\mathbb{P}}}  
\newcommand{\GG}{{\mathbb{G}}}

\newcommand{\YY}{{\mathbb{Y}^5}}  
\newcommand{\F}{{\mathcal{F}}} 

\numberwithin{equation}{section}

\def\colorablu{\textcolor{blue}}

\title[Explicit computations with cubic fourfolds and GM fourfolds]{Explicit computations with cubic fourfolds, Gushel--Mukai fourfolds, and their associated K3 surfaces}

\address{Dipartimento di Matematica e Informatica, Universit\`a degli Studi di Catania} 
\author[G. Staglian\`o]{Giovanni Staglian\`o}
\email{giovanni.stagliano@unict.it}
\date{\today} 

\begin{document}

\begin{abstract} 
We present some applications of the \emph{Macaulay2} software package 
\href{https://faculty.math.illinois.edu/Macaulay2/doc/Macaulay2/share/doc/Macaulay2/SpecialFanoFourfolds/html/index.html}{\emph{SpecialFanoFourfolds}}, 
a package for working with Hodge-special cubic fourfolds and Hodge-special Gushel--Mukai fourfolds. In particular, we show how to construct new examples of such fourfolds, some of which turn out to be rational. We also describe how to calculate K3 surfaces associated with cubic or Gushel-Mukai fourfolds, which relies on an explicit unirationality of some moduli spaces of K3 surfaces.  
\end{abstract}

\maketitle

\section*{Introduction}
One of the aims of this note is to show how some computations 
with  
cubic fourfolds and Gushel--Mukai fourfolds,
which may appear very abstract,
can be done 
in an explicit and completely automatic way.
For this purpose, we will use the package 
\href{https://faculty.math.illinois.edu/Macaulay2/doc/Macaulay2/share/doc/Macaulay2/SpecialFanoFourfolds/html/index.html}{\emph{SpecialFanoFourfolds}} 
\cite{SpecialFanoFourfoldsSource}, which is included with 
\href{https://faculty.math.illinois.edu/Macaulay2/}{\emph{Macaulay2}} \cite{macaulay2} (\footnote{A more updated version of the package is available at \url{https://github.com/Macaulay2/M2/blob/development/M2/Macaulay2/packages/SpecialFanoFourfolds.m2}.}).
This package implements 
several functions to construct and analyze such fourfolds,
revolving around questions of rationality.
Most of the ideas on which these functions are based came from the papers 
\cite{RS1,RS3}; other references are \cite{BRS,HoffSta,HoffSta2,Hanine-Michele-2,famGushelMukai}.
For the general theory on cubic and Gushel--Mukai fourfolds, we mainly refer to \cite{Hassett,Has00,DIM}.

Recall that cubic fourfolds and Gushel--Mukai fourfolds 
are prime Fano fourfolds of index $3$ and $2$, respectively. 
By the work of 
Kobayashi and Ochiai \cite{kobayashi-ochiai}, 
Iskovskikh and Fujita \cite{fujita-polarizedvarieties}, and Mukai \cite{mukai-biregularclassification}, 
prime Fano fourfolds of index $\geq 2$ are completely classified. We summarize in Table~\ref{Table prime Fano fourfolds} this classification.
\begin{table}[htbp]
    \renewcommand{\arraystretch}{0.75}
     \centering
     \tabcolsep=1.0pt 
     \footnotesize   
\begin{adjustbox}{width=\textwidth}
\begin{tabular}{ccccc}
\hline
\rowcolor{gray!5.0}
Fourfold & Index & \begin{tabular}{c} Irrationality of \\ the very general \end{tabular} & \begin{tabular}{c} Description of \\ the rational ones \end{tabular} & Ref. \\
\hline 
\hline 
Projective space $\PP^4$ & $5$ & no & all & trivial \\
\hline
Quadric hypersurface in $\PP^5$ & $4$ & no & all & trivial \\
\hline \rowcolor{green!10}
\begin{tabular}{c} Cubic hypersurface in $\PP^5$ \\ (cubic fourfold) \end{tabular} & $3$ & not known & \begin{tabular}{c} some examples are \\ known but the general \\ picture is not clear \end{tabular} & \cite{Fano} \\
\hline 
\begin{tabular}{c} Complete intersection of \\two  quadrics in $\PP^6$ \end{tabular} & $3$ & no & all & trivial \\
\hline 
\begin{tabular}{c} Linear section in $\PP^{7}$ \\ of $\mathbb{G}(1,4)\subset\PP^{9}$ \end{tabular} & $3$ & no & all & \cite{Roth1949} \\
\hline 
Quartic hypersurface in $\PP^5$ & $2$ & yes & no known examples & \cite{Totaro2015} \\
\hline 
\begin{tabular}{c} Complete intersection of \\ a quadric and a cubic in $\PP^6$ \end{tabular} & $2$ & yes & no known examples & \cite{NOtte} \\
\hline 
\begin{tabular}{c} Complete intersection of \\ three quadrics in $\PP^7$ \end{tabular} & $2$ & yes & \begin{tabular}{c} a countable union of \\  closed subsets in the \\ corresp. moduli space \end{tabular} & \cite{HPT2} \\
\hline \rowcolor{green!10} 
\begin{tabular}{c} Gushel--Mukai fourfold \end{tabular} & $2$ & not known & \begin{tabular}{c} some examples are \\ known but the general \\ picture is not clear \end{tabular} & \cite{Roth1949} \\
\hline 
\begin{tabular}{c} Linear section in $\PP^9$ \\ of the spinorial $\mathbb{S}^{10}\subset\PP^{15}$ \end{tabular}  & $2$ & no & all & \cite{Roth1949} \\
\hline 
\begin{tabular}{c} Linear section in $\PP^{10}$ \\ of $\mathbb{G}(1,5)\subset\PP^{14}$ \end{tabular}  & $2$ & no & all & \cite{Roth1949} \\
\hline 
\begin{tabular}{c} Linear section in $\PP^{11}$  \\ 
of $LG(3,6)\subset\PP^{13}$ \end{tabular} & $2$ & no & all & \cite{Roth1949} \\
\hline 
\begin{tabular}{c} Fourfold of degree $18$ \\ and genus $10$ in $\PP^{12}$ \end{tabular} & $2$ & no & all & \cite{Roth1949} \\
\hline 
\end{tabular}
\end{adjustbox}
\caption{Prime \href{https://faculty.math.illinois.edu/Macaulay2/doc/Macaulay2/share/doc/Macaulay2/SpecialFanoFourfolds/html/_fano__Fourfold.html}{Fano fourfolds} 
of index $\geq 2$. 
The classical constructions of rationality in \cite{Roth1949}
are implemented in the function 
\href{https://faculty.math.illinois.edu/Macaulay2/doc/Macaulay2/share/doc/Macaulay2/MultiprojectiveVarieties/html/_parametrize_lp__Multiprojective__Variety_rp.html}{\texttt{parametrize}}.}
\label{Table prime Fano fourfolds}
\end{table}
For most of the known types of prime Fano fourfolds, 
there are classical constructions 
to produce rational parameterizations 
of the general fourfold.
It is also classically known that some special cubic 
fourfolds and some special Gushel--Mukai fourfolds are rational 
(see \cite{Morin,Fano,Roth1949}). However, the 
question whether 
the general cubic fourfold (resp., the general Gushel--Mukai fourfold) is rational
 or not is still unsolved. 
 It is expected that inside the moduli space of cubic fourfolds (resp., Gushel--Mukai fourfolds),
 the rational fourfolds should belong to the so-called Noether--Lefschetz locus,
 which is a countable union of hypersurfaces; 
 in particular, 
 the very general fourfold should be irrational. 
 Roughly speaking, the Noether--Lefschetz locus parametrizes 
 fourfolds containing a surface whose cohomology 
 class does not come from the ambient space ($\PP^5$ 
 for cubic fourfolds,
 and $\GG(1,4)\subset\PP^9$ for Gushel--Mukai fourfolds).

 One of the useful features of the package
 is the ability to find 
 somewhat random examples of fourfolds
 in the Noether--Lefschetz locus 
 and determine to which component they belong.
 In some likely cases, from 
 an automatic count of parameters
 performed on the example,
 it is possible to get information about the 
 Kodaira dimension of the component. 
 This has already been used in \cite{famGushelMukai} 
 to deduce that the first ten components of 
 the Noether--Lefschetz locus 
 in the moduli space of Gushel--Mukai fourfolds 
 have negative Kodaira dimension. In Theorem~\ref{estensioneTHMruled}
 we will extend this result by including the eleventh component.
 See also Theorem~\ref{Thm unirazionality of Cd}
for known analogous results in the case of cubic fourfolds.

 A second  feature of the package is the
 ability to establish the rationality of many fourfolds
 and  calculate rational parameterizations.
 This is based on ideas introduced in \cite{RS1}.
We will use this feature to show 
new examples of rational Gushel--Mukai fourfolds,
thus continuing the research initiated in
\cite{Roth1949} 
and continued in \cite{DIM,Explicit,HoffSta,rationalSta}.
See Tables~\ref{Table K3 associated with GMs} and \ref{Table K3 associated with GMs (further examples)} for 
a summary of new and old examples of rational Gushel--Mukai fourfolds;
see also Table~\ref{Table K3 associated with cubics} for the case of cubic fourfolds.

The last and most complex feature of the package 
is the possibility of calculating K3 surfaces 
``associated'' with cubic or Gushel--Mukai fourfolds 
(see Subsections~\ref{Subsection associated K3}).
This is based on ideas developed in \cite{RS3} and then applied also in \cite{HoffSta2}.
As a consequence of the construction, one deduces the (explicit) unirationality 
of some moduli space $\mathcal F_g$ of K3 surfaces of genus $g$ 
(see Theorem~\ref{teo esplicity unir}).
Just to give an example of application, the package 
can give us the explicit equations 
of general K3 surfaces of genus $22$ (see Code~Example~\ref{code for F22})
and, 
if we had enough computing power, we could produce 
an explicit dominant rational map from a projective space to 
the moduli space $\mathcal F_{22}$.
The unirationality of 
$\mathcal F_{22}$ 
has recently also been proved by Farkas and Verra \cite{FarkasVerraC42},
but their construction 
seems too abstract to be translated into codes.

The paper is organized as follows.
In section~\ref{Section1}, we recall general facts 
on the Noether--Lefschetz loci in the moduli spaces
of cubic and Gushel--Mukai fourfolds. We also recall the known results 
about the Kodaira dimension of the irreducible components of these loci.
Next we show how to construct new fourfolds. As an application,
we deduce Theorem~\ref{estensioneTHMruled}, 
extending the main result obtained in \cite{famGushelMukai}.
Moreover, we construct new examples of rational Gushel–Mukai fourfolds.
In Section~\ref{Section2}, we first recall the known results 
about the Kodaira dimension of the moduli spaces $\mathcal F_g$ 
of K3 surfaces of genus $g$. Then we recall the constructions 
from \cite{RS3,HoffSta2} 
of the explicit unirationality of $\mathcal F_g$, for $g=11,14,20,22$,
and show how to get equations for general K3 surfaces of these genera.

%\subsection*{Acknowledgements}

\section{Noether--Lefschetz loci in the moduli spaces of cubic fourfolds and Gushel--Mukai fourfolds}\label{Section1}

\subsection{Notation and definitions}
A \emph{cubic fourfold} is  a smooth cubic hypersurface in $ \PP^5 $.
Cubic fourfolds are parametrized by a moduli space $\mathcal{C}$ of dimension $20$.

A \emph{Gushel--Mukai fourfold} (GM fourfold for short) is a smooth quadric hypersurface in 
a $5$-dimensional linear section $\YY\subset \PP^8$ of the 
cone $\widehat{\mathbb{G}(1,4)}\subset\PP^{10}$ over the Grassmannian $\mathbb{G}(1,4)\subset\PP^9$. 
GM fourfolds are parametrized by a moduli space $\mathcal{GM}$ of dimension $24$.
The GM fourfolds $X\subset \YY$ for which $\YY$ contains the vertex of the cone $\widehat{\mathbb{G}(1,4)}$
are called of \emph{Gushel type} and describe 
an irreducible codimension 2 closed subset of $\mathcal{GM}$.
The GM fourfolds which are not of Gushel type are called \emph{ordinary}. These 
 can be realized as
smooth quadric hypersurfaces in a smooth Del Pezzo fivefold $\PP^8\cap \mathbb{G}(1,4)\subset\PP^9$.

The Noether--Lefschetz locus in the moduli space of cubic fourfolds 
has been studied by Hassett (see \cite{Hassett}, see also \cite{Has00,Levico}). It 
is a countable union of irreducible hypersurfaces 
\[\bigcup_{d} \mathcal C_d \subset \mathcal{C}, \quad d\in\{8, 12, 14, 18, 20, 24, 26, 30, 32, 36, 38, 42, 44,\ldots\},\]
where the \emph{discriminant} $d$ runs over all 
integers
$d>6$ with $d\equiv 0,2\,\mathrm{mod}\, 6$.

 The Noether--Lefschetz locus in the moduli space of GM fourfolds 
 has been studied by Debarre, Iliev, and Manivel (see \cite{DIM}, see also \cite{DK1,DK2,DK3}). It
is a countable union of hypersurfaces 
\[\bigcup_{d} \mathcal{GM}_d \subset \mathcal{GM},\quad d\in\{10, 12, 16, 18, 20, 24, 26, 28, 32, 34, 36, 40,\ldots\},\]
where the \emph{discriminant} $d$ runs over all integers 
$d>8$ with $d\equiv 0,2,4\,\mathrm{mod}\,8$. 
If $d\equiv 2\,\mathrm{mod}\,8$ then $\mathcal{GM}_d$
is the union of two irreducible 
components $\mathcal{GM}_d'\cup \mathcal{GM}_d^{''}$,
otherwise it is irreducible. 
Cubic and GM fourfolds belonging to the Noether--Lefschetz locus 
are called Hodge-special. 

\subsection{Formulas for the discriminant}
\label{Section formulas discriminant}
We provide some details on the calculation of the discriminant of 
a Hodge-special fourfold; see also
\cite[Section~4]{Has00} and \cite[Section~7]{DIM}.

Let $[X]\in \mathcal{C}$
be a cubic fourfold containing  an irreducible 
surface $S\subset\PP^5$ 
of degree $\deg(S)$ 
and sectional genus $g(S)$, which has smooth normalization and
 only  a finite number $\delta$ 
of nodes as singularities. 
Then one computes (see \cite[Theorem~9.3]{fulton-intersection}) that the self-intersection of $S$ in $X$ is given by
\begin{equation}\label{autointersezione caso cubic fourfolds}
(S)_X^2 = 3\,\deg(S) + 6\,g(S) - 12\,\chi(\mathcal O_S) + 2\,K_S^2 + 2\,\delta - 6 .
\end{equation}
Denoting by $H_X$ the class of the hyperplane section, we have 
 that $[X]\in \mathcal{C}_d$, 
where $d$ is 
the discriminant 
of the lattice 
spanned by $([H_X^2],[S])$, that is
\begin{equation}\label{discriminante caso cubic fourfolds}
 d = \mathrm{disc}\left[
\begin{array}{c|cc}
     & H_X^2 & [S] \\ \hline 
     H_X^2 & 3 & \deg(S) \\ \newline
     [S] & \deg(S) & (S)_X^2
     \end{array}
     \right]= 3(S)_X^2 - \deg(S)^2 .
\end{equation}

Similarly, let  
  $[X\subset\YY]\in \mathcal{GM}$ 
  be a GM fourfold
  containing an irreducible 
surface $S\subset\YY$,
which has smooth normalization and  
only  
a finite number $\delta$ 
of nodes as singularities. Let 
 $a\sigma_{3,1} + b\sigma_{2,2}$ denote the class of $\gamma_{X\ast}(S)$ 
 in the Chow ring of $\mathbb{G}(1,4)$,
 where $\gamma_X:X  \rightarrow \mathbb{G}(1,4)$ 
 is the so-called \emph{Gushel map}, defined 
 as the composition of the 
 embedding of $X$ into the cone $\widehat{\mathbb{G}(1,4)}$
followed by the projection from 
 the vertex.
 Again one computes that the self-intersection of $S$ in $X$ is given by
\begin{equation}\label{autointersezioneGM}
(S)_X^2 =
 a + 2\,b + 4\,g(S) - 12\,\chi(\mathcal O_S) + 2\,K_S^2 + 2\,\delta - 4 .
\end{equation}
We have that $[X]\in \mathcal{GM}_d$, 
where $d$ is 
the discriminant of the lattice 
spanned by 
$(\gamma_X^\ast(\sigma_{1,1}), \gamma_X^{\ast}(\sigma_{2}), [S])$, that is
\begin{equation}\label{discriminanteGM}
 d=\mathrm{disc}\left[
\begin{array}{c|ccc}
    & \gamma_X^\ast(\sigma_{1,1}) & \gamma_X^{\ast}(\sigma_{2}) & [S] \\ \hline 
     \gamma_X^\ast(\sigma_{1,1}) & 2 & 2 & b \\ \newline
      \gamma_X^{\ast}(\sigma_{2}) & 2 & 4 & a \\ \newline 
     \left[S\right] & b & a & (S)_X^2 \\
     \end{array}
     \right]= 4(S)_X^2 -2 a^{2}+4 a b-4b^{2} .
\end{equation}
If $d\equiv 2 \,\mathrm{mod}\,8$, then
$[X]\in\mathcal {GM}_d'$ 
if $a+b$ is even, and $[X]\in\mathcal {GM}_d^{''}$ 
if $b$ is even 
(see \cite[Corollary~6.3]{DIM}. 

\subsection{Implementation}
In the \emph{Macaulay2} package 
\href{https://faculty.math.illinois.edu/Macaulay2/doc/Macaulay2/share/doc/Macaulay2/SpecialFanoFourfolds/html/index.html}{\emph{SpecialFanoFourfolds}},
Hodge-special fourfolds are implemented 
 as subtypes of the class 
\href{https://faculty.math.illinois.edu/Macaulay2/doc/Macaulay2/share/doc/Macaulay2/MultiprojectiveVarieties/html/___Embedded__Projective__Variety.html}{``\texttt{embedded projective variety}''},
provided by the package \href{https://faculty.math.illinois.edu/Macaulay2/doc/Macaulay2/share/doc/Macaulay2/MultiprojectiveVarieties/html/index.html}{\emph{MultiprojectiveVarieties}} \cite{2021package}. However, a Hodge-special fourfold is represented internally by a pair 
 $(S,X)$, where 
$X$ is the fourfold and $S$ is a  
special hidden surface contained in $X$.
The command \href{https://faculty.math.illinois.edu/Macaulay2/doc/Macaulay2/share/doc/Macaulay2/SpecialFanoFourfolds/html/_surface.html}{\texttt{surface(X)}} returns this surface $S$.
Hodge-special fourfolds can be created by the two functions: 
\href{https://faculty.math.illinois.edu/Macaulay2/doc/Macaulay2/share/doc/Macaulay2/SpecialFanoFourfolds/html/_special__Cubic__Fourfold.html}{\texttt{specialCubicFourfold}} and
\href{https://faculty.math.illinois.edu/Macaulay2/doc/Macaulay2/share/doc/Macaulay2/SpecialFanoFourfolds/html/_special__Gushel__Mukai__Fourfold.html}{\texttt{specialGushelMukaiFourfold}},
which typically expect the pair $(S,X)$ as input and always return $X$.
It is also possible to give only the 
surface $S$ embedded in $\PP^5$ or in a fivefold $\YY$, so that
\emph{Macaulay2}
will randomly choose a fourfold $X$ containing $S$.
In the case when $X$ is a GM fourfold given as a subvariety of $\mathbb{P}^8$,
the {Gushel map} of $X$ is 
 calculated behind the scenes
 and can be retrieved with the command
 \href{https://faculty.math.illinois.edu/Macaulay2/doc/Macaulay2/share/doc/Macaulay2/SpecialFanoFourfolds/html/_to__Grass.html}{\texttt{toGrass(X)}}.
The function \href{https://faculty.math.illinois.edu/Macaulay2/doc/Macaulay2/share/doc/Macaulay2/SpecialFanoFourfolds/html/_discriminant_lp__Special__Gushel__Mukai__Fourfold_rp.html}{\texttt{discriminant}}, as its name suggests, calculates
the discriminant of the fourfold. This is done by applying the formulas 
\eqref{autointersezione caso cubic fourfolds}, \eqref{discriminante caso cubic fourfolds},
\eqref{autointersezioneGM}, and \eqref{discriminanteGM}, where the value of $K_S^2$ 
is determined by functions from the packages 
\href{https://faculty.math.illinois.edu/Macaulay2/doc/Macaulay2/share/doc/Macaulay2/Cremona/html/index.html}{\emph{Cremona}} \cite{packageCremona} 
and 
\href{https://faculty.math.illinois.edu/Macaulay2/doc/Macaulay2/share/doc/Macaulay2/CharacteristicClasses/html/index.html}{\emph{CharacteristicClasses}} \cite{Jost2015}, and 
the normalization of $S$ via the package \href{https://faculty.math.illinois.edu/Macaulay2/doc/Macaulay2/share/doc/Macaulay2/IntegralClosure/html/index.html}{\emph{IntegralClosure}} \cite{IntegralClosureSource}.
In some internal calculations it is required to invert complicated birational maps.
For this, functions from the package \href{https://faculty.math.illinois.edu/Macaulay2/doc/Macaulay2/share/doc/Macaulay2/RationalMaps/html/index.html}{\emph{RationalMaps}} \cite{RationalMaps} are also used.
\begin{codeexample}
In the following code, 
 we create 
a GM fourfold 
containing a \emph{$\tau$-quadric surface},
that is, a two-dimensional linear
section of a Schubert variety $\Sigma_{1,1}\simeq\GG(1,3)\subset\GG(1,4)$ (first row of Table~\ref{Table K3 associated with GMs}).
We input the equations in the ring of polynomials with $9$ variables $a,b,\ldots,i$.
Note that by default, several checks are done on the input data; 
these can be relaxed or strengthened using the option 
\href{https://faculty.math.illinois.edu/Macaulay2/doc/Macaulay2/share/doc/Macaulay2/SpecialFanoFourfolds/html/index.html}{\texttt{InputCheck}}.   
\begin{tcolorbox}[breakable=true,boxrule=0.5pt,opacityback=0.1,enhanced jigsaw]  
{\footnotesize
\begin{Verbatim}[commandchars=&!$]
&colore!darkorange$!M2 -q --no-preload$
&colore!output$!Macaulay2, version 1.19.1$
&colore!darkorange$!i1 :$ &colore!airforceblue$!needsPackage$ "&colore!bleudefrance$!SpecialFanoFourfolds$"; &colore!commentoutput$!-- version 2.5.1$ 
&colore!darkorange$!i2 :$ &colore!airforceblue$!QQ$[&colore!airforceblue$!vars$(0..8)]; &colore!commentoutput$!-- coordinate ring of PP^8$  
&colore!darkorange$!i3 :$ S = &colore!airforceblue$!projectiveVariety ideal$(i, f, c, b, a, e*g-d*h);
&colore!circOut$!o3 :$ &colore!output$!ProjectiveVariety, surface in PP^8$
&colore!darkorange$!i4 :$ X = &colore!airforceblue$!projectiveVariety ideal$(e*g-d*h+b*i, e*f-c*h+a*i, 
         d*f-c*g-a*i-b*i-c*i-f*i-i^2, b*f-a*g-a*h-b*h-c*h-f*h-h*i, 
         b*c-a*d-a*e-b*e-c*e-c*h+a*i-e*i, 2*a*b+b^2+2*a*c+c^2+2*a*d-c*d+
             a*e+2*b*e+2*c*e+2*a*f+2*c*f+f^2+3*a*g+b*g+2*c*g+f*g+2*a*h+b*h+
                 3*c*h+2*f*h+a*i+3*b*i+3*c*i-d*i+3*e*i+3*f*i+g*i+2*h*i+2*i^2);
&colore!circOut$!o4 :$ &colore!output$!ProjectiveVariety, 4-dimensional subvariety of PP^8$
&colore!darkorange$!i5 :$ X = &colore!bleudefrance$!specialGushelMukaiFourfold$(S,X);
&colore!circOut$!o5 :$ &colore!output$!ProjectiveVariety, GM fourfold containing a surface$
&colore!output$!     of degree 2 and sectional genus 0$
&colore!darkorange$!i6 :$ &colore!bleudefrance$!surface$ X
&colore!circOut$!o6 =$ &colore!output$!S$
&colore!circOut$!o6 :$ &colore!output$!ProjectiveVariety, surface in PP^8$
&colore!darkorange$!i7 :$ &colore!bleudefrance$!describe$ X
&colore!circOut$!o7 =$ &colore!output$!Special Gushel-Mukai fourfold of discriminant 10(')$
&colore!output$!     containing a surface in PP^8 of degree 2 and sectional genus 0$
&colore!output$!     cut out by 6 hypersurfaces of degrees (1,1,1,1,1,2)$
&colore!output$!     and with class in G(1,4) given by s_(3,1)+s_(2,2)$
&colore!output$!     Type: ordinary$
\end{Verbatim}
} \noindent 
\end{tcolorbox}
\end{codeexample}     

\subsection{Kodaira dimension of the components}\label{Subsection Kodaira Cd}
We recall known results about the Kodaira dimension of the components
of the 
Noether--Lefschetz loci in the moduli spaces $\mathcal C$
and $\mathcal{GM}$.
The first result in this direction has been obtained in \cite{Nuer},
by showing the following theorem with the exception of $d=42$.
 \begin{theorem}[\cite{Nuer,Lai,FarkasVerraC42}]
     \label{Thm unirazionality of Cd}
  Each irreducible component $\mathcal C_d$ of the Noether--Lefschetz locus in $\mathcal{C}$
 is {unirational} if the discriminant $d$ is at most $44$.
 \end{theorem}
 On the other side we have the following.
 \begin{theorem}[\cite{TaVaA}] \hspace{1pt}
     \label{non unirazionalita di Cd}
  \begin{itemize} 
   \item The component $\mathcal C_d$ 
   is of \emph{general type} 
   for any $d\geq 114$,
 with the possible exceptions $d\in\{120, 122, 128, 132, 138, 150, 152, 180, 192\}$.
   \item The Kodaira dimension $\kappa(\mathcal C_d)$ is {non-negative} for 
    any $d\geq 86$,
   with the possible exceptions $d\in\{90, 92, 96, 108, 120, 132, 180\}$.
  \end{itemize}
 \end{theorem}
 In the case of GM fourfolds we have the following weaker results.
 \begin{theorem}[\cite{famGushelMukai}]\label{uniruledness-GM}
 Each irreducible component of the Noether--Lefschetz locus in $\mathcal{GM}$
 is {uniruled} if the discriminant $d$ is at most $26$;
  moreover, $\mathcal{GM}_{10}'$, $\mathcal{GM}_{10}^{''}$, $\mathcal{GM}_{12}$, and $\mathcal{GM}_{20}$ are {unirational}.
 \end{theorem}
 In Theorem~\ref{estensioneTHMruled}
 we will extend Theorem~\ref{uniruledness-GM} by including the case $d=28$.
\begin{remark}
From the main result of \cite{Ma2018-oc} (see also \cite{petok2021kodaira}),
 it follows that 
only a finite number of components of the  
Noether--Lefschetz locus in $\mathcal{GM}$
are unirational (in such case the discriminant is at most $224$).
Nevertheless, it does not yet seem clear if the same holds true 
about the uniruledness.
\end{remark}
Theorem~\ref{uniruledness-GM} is proved in \cite{famGushelMukai} by constructing very particular 
examples of Hodge-special GM fourfolds; more precisely, examples of
pairs $(S,X)$, where $S\subset \YY$ is a smooth irreducible surface in a smooth Del Pezzo fivefold $\YY$,
and $[X]\in\PP(H^0(\mathcal{O}_{\YY}(2)))\simeq\PP^{39}$ is a smooth hyperquadric in $\YY$ containing $S$ (some of the most relevant examples are included
 in Table~\ref{Table uniruledness GMd d small}). 
Then, 
by applying Proposition~\ref{Prop parameterCount} below,
which can be done automatically with 
the function \href{https://faculty.math.illinois.edu/Macaulay2/doc/Macaulay2/share/doc/Macaulay2/SpecialFanoFourfolds/html/_parameter__Count.html}{\texttt{parameterCount}},
we deduce that 
there exists a family $\mathcal S\subset \mathrm{Hilb}_{\YY}^{\chi(\mathcal{O}_S(t))}$ 
of surfaces with $[S]\in \mathcal S$ such that
the closure of the  locus of smooth hyperquadrics in $\YY$
      containing some 
      surface of $\mathcal S$
      describes an irreducible component of the Noether--Lefschetz locus in $\mathcal{GM}$.
      This component is uniruled since it is covered by curves birational to pencils
      of GM fourfolds through surfaces of $\mathcal S$ (indeed we have $h^0(\mathcal{I}_{S/\YY}(2))>1$ for each $[S]\in\mathcal S$).
\begin{proposition}[Count of parameters]
\label{Prop parameterCount}
Let $\mathbb{V}$ be either $\PP^5$ or a Del Pezzo fivefold 
 $\YY=\GG(1,4)\cap \PP^8$, and 
let $r$ be respectively equal to  $3$ and $2$.
Let $S\subset \mathbb V$ be a smooth irreducible surface
which is contained in a smooth hypersurface $X\subset\mathbb{V}$ of degree $r$.
Assume that 
\begin{enumerate}
\item $h^1(N_{S/\mathbb{V}})=0$, and
\item $h^1(\mathcal O_S(r))=0$ and $h^0(\mathcal{I}_{S/\mathbb{V}}(r))=h^0(\mathcal{O}_{\mathbb V}(r))-\chi(\mathcal O_S(r))$.
\end{enumerate}     
Then 
there is a unique irreducible component $\mathcal S\subset\mathrm{Hilb}(\mathbb V)$ 
of the Hilbert scheme 
of $\mathbb V$ that contains $[S]$, and the family $\mathfrak X_{\mathcal S}\subset \PP(H^0(\mathcal O_{\mathbb V}(r)))$ of 
the hypersurfaces in $\mathbb V$ of degree $r$ containing some surface 
of the family $\mathcal S$ has codimension at most 
\begin{equation*}
\dim(\mathbb{P}(H^0(\mathcal O_{\mathbb V}(r)))) - \left(h^0(N_{S/\mathbb V}) + h^0(\mathcal I_{S/\mathbb V}(r)) - h^0(N_{S/X}) -1\right).
\end{equation*}
Furthermore,  if this last value is $1$ and by applying the formulas in Subsection~\ref{Section formulas discriminant}
 we get a non-zero value of $d$,
then $\mathfrak X_{\mathcal S}$ is a hypersurface; 
after passing to the quotient modulo  $\mathrm{Aut}(\mathbb V)$,
this  gives rise to an irreducible component 
of the Noether--Lefschetz locus 
parameterizing fourfolds of discriminant $d$.
\end{proposition}   
\begin{proof}[Sketch of the proof] See 
     also \cite{Nuer} and \cite[Subsection~1.5]{famGushelMukai}.
     By the condition (1), we deduce that the surface $S$ 
     corresponds to a smooth point $[S]$ of the Hilbert scheme $\mathrm{Hilb}(\mathbb V)$. Therefore 
     there exists a unique irreducible component $\mathcal S$ of $\mathrm{Hilb}(\mathbb V)$ that contains $[S]$.
     Let \[\mathcal{X}_{\mathcal S}=\{([S'],[X']):S'\subset X'\}\subset \mathcal{S} \times \PP(H^0(\mathcal{O}_{\mathbb V}(r)))\] be the incidence correspondence.
     The fiber at a point $[S']\in\mathcal S$ 
     of the first projection $\pi_1:\mathcal{X}_{\mathcal S}\to \mathcal S$ 
     is isomorphic to the linear space $\PP(H^0(\mathcal{I}_{S'/\mathbb{V}}(r)))$. 
     By the semicontinuity theorem, we have that the dimension 
     of the fiber of $\pi_1$ achieves its minimum value on an open set of $\mathcal S$,
     and by the condition (2) it follows that the point $[S]$ belongs to this open set.
     Therefore there exists a unique irreducible component $\mathcal Z$ of $\mathcal X_{\mathcal S}$
     that dominates $\mathcal S$, and its dimension is equal to 
     \[\dim(\mathcal Z)=\dim(\mathcal S)+h^0(\mathcal{I}_{S/\mathbb{V}}(r))-1=
     h^0(N_{S/\mathbb V})+h^0(\mathcal{I}_{S/\mathbb{V}}(r))-1 . \]
Now the fiber at the point $[X]\in \mathbb{P}(H^0(\mathcal O_{\mathbb V}(r)))$ 
of the restriction to $\mathcal Z$ of the second projection $\pi_2|_{\mathcal Z}:\mathcal Z\subseteq\mathcal{X}_{\mathcal S}\to \mathbb{P}(H^0(\mathcal O_{\mathbb V}(r)))$
contains the point $[S,X]$,
    and we have $\dim_{[S,X]}(\pi_2|_{\mathcal Z}^{-1}([X]))\leq 
    \dim_{[S]}(\mathrm{Hilb}(X))\leq h^0(N_{S/X})
    $. 
    By semicontinuity, we deduce 
    that the generic fiber of $\pi_2|_{\mathcal Z}$ has dimension at most $h^0(N_{S/X})$.
    Thus we have
    \[\dim(\pi_2(\mathcal X_{\mathcal S}))\geq 
       \dim(\pi_2(\mathcal Z))\geq 
    \dim(\mathcal Z)-h^0(N_{S/X}) . \]
    Finally, if from the formulas in Subsection~\ref{Section formulas discriminant}
    we get a non-zero value of $d$, then we deduce that 
$\pi_2(\mathcal X_{\mathcal S})$ 
 does not fill the whole space since 
 its image in the corresponding moduli space is contained in the Noether--Lefschetz locus.
\end{proof}  

\subsection{Finding new examples of Hodge-special fourfolds}
We 
present 
 a slight simplification of a construction used in \cite{famGushelMukai}
that allows us to find new examples of Hodge-special GM fourfolds
and extend the main result obtained there.
We start by fixing 
a birational transformation
$\psi:\PP^6\dashrightarrow\GG(1,4)\subset\PP^9$,
which we choose to be the 
inverse of the projection from a $\sigma_{2,2}$-plane, that is, the
transformation defined 
by the linear system of quadrics
through a Segre threefold
$\Sigma_3\simeq\PP^1\times\PP^2\subset \PP^5\subset \PP^6$.
Now we want to take some good surface $S\subset\mathbb{G}(1,4)$ (possibly smooth and cut out by quadrics),
obtained as the image via $\psi$ of some  surface $T\subset\PP^6$,
and get a GM fourfold as the transversal intersection of $\GG(1,4)$ 
with a hyperplane and a hyperquadric through $S$.
If $T$
does not cut the base locus $\Sigma_3$ of $\psi$, 
we typically obtain uninteresting surfaces in $\GG(1,4)$ 
which are neither contained in a hyperplane of $\PP^9$ nor in a hyperquadric 
of $\GG(1,4)$. Therefore we take surfaces $T\subset \PP^6$
 together with an automorphism $\sigma$ 
of $\PP^6$ sending  a curve $C$ on $T$ to another curve on $\Sigma_3$. 
By replacing $T$ by $\sigma(T)$, we obtain 
a surface cutting $\Sigma_3$ at least along a curve,
and we can hope the image $S=\overline{\psi(T)}$
is good enough for our purposes.
In practice, this
can be achieved automatically only if
 $T$ is some simple rational surface 
 such that we are able to find curves  $C\subset T$ of low degree and low genus. 
So we restrict ourselves to a simpler case.
Let $T=T(a;i,j,k,\ldots)\subset\PP^6$ be the  rational surface 
obtained as the image of the plane via the linear system 
of curves of degree $a$ having 
$i$ general base points $p_1^1,\ldots,p_i^1$ of multiplicity $1$,
$j$ general base points $p_1^2,\ldots,p_j^2$ of multiplicity $2$,
$k$ general base points $p_1^3,\ldots,p_k^3$ of multiplicity $3$,
and so on. Let $C=C(e;l,m,n,\dots)\subset T$ be the curve which is represented in the plane 
by a general curve of degree $e$ passing through 
$l$ of the $i$ points $p_1^1,\ldots,p_i^1$, $m$ 
of the $j$ points $p_1^2,\ldots,p_j^2$, 
$n$ of the $k$ points $p_1^3,\ldots,p_k^3$, and so on.
Suppose we are able to get an automorphism $\sigma$ of $\PP^6$ 
which sends $C$ into a curve $\sigma(C)\subset \Sigma_3$, and that 
the image 
$S=\overline{\psi(\sigma(T))}\subset\GG(1,4)$ is a surface 
contained in a (smooth) GM fourfold. We 
denote by $\mathcal G^{a,i,j,k,\ldots}_{e,l,m,n,\ldots}$ 
the GM fourfold 
obtained as the intersection of $\GG(1,4)$ with a general hyperplane and 
a general hyperquadric through $S$ (we leave out the dependence on $\sigma$).

\subsubsection{Running the construction}
Using the package \emph{SpecialFanoFourfolds}
we can perform the above construction 
of the GM fourfold $\mathcal G^{a,i,j,k,\ldots}_{e,l,m,n,\ldots}$
by giving just one command:
\begin{tcolorbox}[breakable=true,boxrule=0.5pt,opacityback=0.1,enhanced jigsaw]  
\begin{center}
\href{https://faculty.math.illinois.edu/Macaulay2/doc/Macaulay2/share/doc/Macaulay2/SpecialFanoFourfolds/html/_special__Gushel__Mukai__Fourfold_lp__Array_cm__Array_cm__String_cm__Thing_rp.html}{\texttt{specialGushelMukaiFourfold([a,i,j,k,...],[e,l,m,n,...])}}
\end{center}
\end{tcolorbox}
\noindent 
Optionally, we can specify the coefficient ring as second argument (a large finite field is used by default).
Everything is done automatically, including the searching of the automorphism 
$\sigma$
and the needed checks on the surface $S=\overline{\psi(\sigma(T))}$.
If no error message occurs, then we can conclude that the GM fourfold 
$\mathcal G^{a,i,j,k,\ldots}_{e,l,m,n,\ldots}$ exists and,
most importantly, has been successfully created.
We refer to the online documentation for more details.
See Tables~\ref{Table uniruledness GMd d small} ,\ref{Table uniruledness GM28}, \ref{Table K3 associated with GMs}, and \ref{Table K3 associated with GMs (further examples)} for some examples where this procedure works well.
These examples have been found by another function provided by the package (available only in debug mode) that automatically scans many combinations 
of pairs of lists of integers
$((a,i,j,k,\ldots),(e,l,m,n,\ldots))$.

\subsubsection{Explicit cubic fourfolds of high discriminant}
As an application, we can construct Hodge-special cubic fourfolds of high discriminant.
This is just one of the examples.
The projection of the surface $S\subset\YY$
     as in the second row of Table~\ref{Table K3 associated with GMs (further examples)}
     from a general plane in $\YY$ of type $\sigma_{2,2}$
     yields a surface $R\subset\PP^5$ of degree $14$ 
     and sectional genus $7$,
     cut out by $3$ cubics, $6$ quartics, and $2$ quintics,
     having $15$ nodes as the only singularities,
     and $S$ as its normalization.
     A general cubic fourfold $Z\subset\PP^5$ containing $R$ is smooth, 
     and by \eqref{autointersezione caso cubic fourfolds} 
     and \eqref{discriminante caso cubic fourfolds}
     we deduce that $[Z]\in\mathcal{C}_{86}$.
     This works in other cases. For instance, 
     starting from the surface $S\subset\YY$ 
     as in the fifth row of Table~\ref{Table K3 associated with GMs},
     we get a cubic fourfold in $\mathcal{C}_{62}$.

\subsection{Geometric description of the component \texorpdfstring{$\mathcal{GM}_{28}$}{GM28}}
We point out that the existence of examples of GM fourfolds 
as in Table~\ref{Table uniruledness GM28} 
extends Theorem~\ref{uniruledness-GM} to another case.
More precisely, we have the following.
\begin{theorem}
\label{estensioneTHMruled}
The component $\mathcal{GM}_{28}$ is uniruled. In particular, 
each irreducible component of the Noether--Lefschetz locus in $\mathcal{GM}$
has negative Kodaira dimension
 if the discriminant $d$ 
 is at most $28$.
\end{theorem}  
The proof of this result follows from the computation performed in 
 Code~Example~\ref{code componente GM28}
below (see Proposition~\ref{Prop parameterCount}). 
Indeed it tells us that 
there exists a  generically smooth, $27$-dimensional, irreducible component $\mathcal S$ of the Hilbert scheme of 
a smooth Del Pezzo fivefold $\YY=\PP^8\cap\GG(1,4)$,
whose general point corresponds to a 
 smooth rational surface of degree 
$11$ and sectional genus $4$ and with class $6\sigma_{3,1}+5\sigma_{2,2}$
in $\GG(1,4)$.
The family of hyperquadrics in $\YY$ through some surface of $\mathcal S$
is a hypersurface in $\PP(H^0(\mathcal O_{\YY}(2)))\simeq\PP^{39}$,
which gives rise to the component $\mathcal{GM}_{28}$ by the formulas \eqref{autointersezioneGM} and \eqref{discriminanteGM}. Moreover, 
a general fourfold in $\mathcal{GM}_{28}$ contains only a finite number of surfaces of $\mathcal S$,
and  the family 
of hyperquadrics through a general $[S]\in\mathcal S$ is a projective space of dimension $11$.
Note, however, that we are unable to compute the generic members of $\mathcal S$ and of $\mathcal{GM}_{28}$.
\begin{codeexample}[Geometric description of $\mathcal{GM}_{28}$]
     \label{code componente GM28} In the code we construct a Hodge-special GM fourfold $(S,X)$ as 
    described in the last row of Table~\ref{Table uniruledness GM28}. 
    The function 
    \href{https://faculty.math.illinois.edu/Macaulay2/doc/Macaulay2/share/doc/Macaulay2/SpecialFanoFourfolds/html/_parameter__Count.html}{\texttt{parameterCount}} gives 
the three
    invariants 
    $h^0(\mathcal I_{S/\YY}(2))$, $h^0(N_{S/\YY})$, and $h^0(N_{S/X})$ reported in the table. 
\begin{tcolorbox}[breakable=true,boxrule=0.5pt,opacityback=0.1,enhanced jigsaw]     
 {\footnotesize
\begin{Verbatim}[commandchars=&!$]
&colore!darkorange$!i8 :$ &colore!darkorchid$!time$ X = &colore!bleudefrance$!specialGushelMukaiFourfold$([7,0,6,2],[2,0,5,0]);
     &colore!darkorchid$!-- used 10.3554 seconds$
&colore!circOut$!o8 :$ &colore!output$!ProjectiveVariety, GM fourfold containing a surface$
&colore!output$!     of degree 11 and sectional genus 4$
&colore!darkorange$!i9 :$ &colore!darkorchid$!time$ &colore!bleudefrance$!describe$ X
     &colore!darkorchid$!-- used 2.22773 seconds$
&colore!circOut$!o9 =$ &colore!output$!Special Gushel-Mukai fourfold of discriminant 28$
&colore!output$!     containing a surface in PP^8 of degree 11 and sectional genus 4$
&colore!output$!     cut out by 17 hypersurfaces of degree 2$
&colore!output$!     and with class in G(1,4) given by 6*s_(3,1)+5*s_(2,2)$
&colore!output$!     Type: ordinary$
&colore!darkorange$!i10 :$ &colore!darkorchid$!time$ X = &colore!bleudefrance$!parameterCount$(X,&colore!bleudefrance$!Verbose$=>&colore!airforceblue$!true$)
&colore!output$!      -- h^1(N_{S,Y}) = 0$
&colore!output$!      -- h^0(N_{S,Y}) = 27$
&colore!output$!      -- h^1(O_S(2)) = 0, and h^0(I_{S,Y}(2)) = 12 = h^0(O_Y(2)) - \chi(O_S(2));$
&colore!output$!      -- in particular, h^0(I_{S,Y}(2)) is minimal$
&colore!output$!      -- h^0(N_{S,X}) = 0$
&colore!output$!      -- codim{[X] : S ⊂ X ⊂ Y} <= 1$
      &colore!darkorchid$!-- used 1164.59 seconds$
&colore!circOut$!o10 =$ &colore!output$!(1, (12, 27, 0))$
\end{Verbatim}
} \noindent 
\end{tcolorbox}
\end{codeexample}

\subsection{Hodge-associated K3 surfaces}\label{Subsection associated K3}
For some infinitely many values of
 the discriminant $d$, 
a fourfold $[X]\in \mathcal C_d$ (resp., $[X]\in\mathcal{GM}_d$)
has a \emph{(Hodge-)associated K3 surface} of degree $d$ and genus $\frac{d}{2}+1$. 
We call such values $\mathcal C$-admissible (resp., $\mathcal{GM}$-admissible).
Table~\ref{TableAdmissible} reports 
the first $\mathcal C$- and $\mathcal{GM}$-admissible values;
see \cite[Theorem~1.0.2]{Has00} and \cite[Proposition~6.5]{DIM} for the precise definitions.
\begin{table}[htbp]
\centering
%\tabcolsep=1.5pt 
%\begin{adjustbox}{width=\textwidth}
\begin{tabular}{r|cccccccccccccccccc}
\hline 
  $\mathcal C$ &  8 &  & 12 & \cellcolor{green!20}\colorablu{14} &  & 18 & 20 & 24 & \cellcolor{green!20}\colorablu{26} &  & 30 & 32 & & 36 & \cellcolor{green!20}\colorablu{38} &  & \cellcolor{green!20}\colorablu{42} & 44 \\
 $\mathcal{GM}$ &  & \cellcolor{green!20}\colorablu{10} & 12 &  & 16 & 18 & \cellcolor{green!20}\colorablu{20} & 24 & \cellcolor{green!20}\colorablu{26} & 28 &  & 32 & \cellcolor{green!20}\colorablu{34} & 36 &  & 40 & 42 & 44 \\
 \hline 
 \\ \hline 
   $\mathcal C$ & 48 & 50 &  & 54 & 56 &  & 60 & \cellcolor{green!20}\colorablu{62} & & 66 & 68 & 72 & \cellcolor{green!20}\colorablu{74} & & \cellcolor{green!20}\colorablu{78} & 80 & & 84 \\
   $\mathcal{GM}$ & 48 & \cellcolor{green!20}\colorablu{50} & \cellcolor{green!20}\colorablu{52} &  & 56 & \cellcolor{green!20}\colorablu{58} & 60 & & 64 & 66 & \cellcolor{green!20}\colorablu{68} & 72 & \cellcolor{green!20}\colorablu{74} & 76 & & 80 & \cellcolor{green!20}\colorablu{82} & 84\\
 \hline 
 \end{tabular}
     %\end{adjustbox}
     \caption{$\mathcal C$- and $\mathcal{GM}$-admissible values $<86$.}
     \label{TableAdmissible}
\end{table}

\begin{remark}[Geometric interpretation of associated K3 surfaces]
An integer $d$ of the form $2(n^2+n+1)/a^2$ for some $(a,n)\in\mathbb{Z}^2$ 
is always $\mathcal C$-admissible, but the converse is not true.
For instance, we have
      $14 = \frac{2(2^2+2+1)}{1^2}$, $26 = \frac{2(3^2+3+1)}{1^2}$, $38 = \frac{2(30^2+30+1)}{7^2}$, 
      $42 = \frac{2(4^2+4+1)}{1^2}$, $62 = \frac{2(5^2+5+1)}{1^2}$, but the next $\mathcal C$-admissible value $74$
      is not of this form.
Addington \cite{Add2016} (see also \cite{Hassett} and \cite{Ouchi2020})
proved that if $[X]$ is a cubic fourfold in $\mathcal C_d$, then 
the following are equivalent:
\begin{itemize}
       \item $d$ is of the form $2(n^2+n+1)/a^2$, for some $(a,n)\in \mathbb{Z}^2$;
       \item the Fano variety $F(X)$ of lines in $X$ 
       is birational to the Hilbert scheme $\mathrm{Hilb}^2(U)$ 
       of two points on some 
       K3 surface $U$ of degree $d$ (necessarily associated with $X$).
      \end{itemize}
Moreover,
if 
$a=1$ and $[X]\in \mathcal C_d$ is general, 
we have 
an isomorphism $F(X)\simeq \mathrm{Hilb}^2(U)$.
\end{remark}    

The notion of associated K3 surface leads 
to the following rationality conjecture (see \cite{kuz4fold,AT,kuz2,Levico,DIM}). 
\begin{conjecture}[Kuznetsov's conjecture]
A fourfold $[X]\in \mathcal C$ is rational  if and only if 
$[X]\in\mathcal C_d$ for some $\mathcal C$-admissible value $d$, that is
\[[X]\in \mathcal C_{14}\cup \mathcal C_{26}\cup \mathcal C_{38} 
\cup \mathcal C_{42} \cup \mathcal C_{62} \cup \mathcal C_{74} \cup \mathcal C_{78} \cup \mathcal C_{86}\cup\cdots\]

A fourfold $[X]\in\mathcal{GM}$ is rational 
 if and only if $[X]\in\mathcal{GM}_d$ for some $\mathcal{GM}$-admissible value $d$,
 that is,
 \[[X]\in \mathcal{GM}_{10}'\cup\mathcal{GM}_{10}^{''}\cup\mathcal{GM}_{20} 
 \cup \mathcal{GM}_{26}'\cup\mathcal{GM}_{26}^{''} \cup\mathcal{GM}_{34}^{'}\cup\cdots \]
\end{conjecture}
It is classically known that cubic fourfolds in $\mathcal{C}_{14}$
are rational (see \cite{Fano,Morin}, see also \cite{BRS}),
as well as that GM fourfolds in $\mathcal{GM}_{10}^{''}$ are rational 
(see \cite{Roth1949,Enr}, see also \cite{DIM}).
The following result summarizes the current state of the conjecture.
\begin{theorem}[``first cases'' + \cite{RS1,RS3,HoffSta}]\label{solutionKuz}
The fourfolds in $\mathcal C_{14}\cup \mathcal C_{26}\cup \mathcal C_{38} \cup \mathcal C_{42} \bigcup  
\mathcal{GM}_{10}'\cup\mathcal{GM}_{10}^{''}\cup\mathcal{GM}_{20}$ are rational.
\end{theorem}
   
\begin{remark}
  Although we 
  are 
  unable 
  to prove that 
  every fourfold in $\mathcal{GM}_{26}$ is rational, 
  we can exhibit several examples of ordinary rational fourfolds in $\mathcal{GM}_{26}^{''}$. 
  We include two of these examples
  in the last two rows of Table~\ref{Table K3 associated with GMs}.
  The last row contains an entirely new example. 
  The penultimate line contains 
  a new example of ordinary fourfold, but 
  examples 
  of non-ordinary fourfolds containing the same surface  
  were already constructed in 
  \cite{rationalSta}.
 See Subsection~\ref{quarta riga della tabella} below for more details 
  on the example of the last row of Table~\ref{Table K3 associated with GMs}.
  See also Table~\ref{Table K3 associated with GMs (further examples)} 
  for some example of rational fourfold in $\mathcal{GM}_{34}$.
\end{remark}

\section{Moduli spaces of K3 surfaces}\label{Section2}
 The moduli space $\mathcal F_g$ of polarized K3 surfaces of genus $g$
 parametrizes pairs $(S,H)$, where $S$ is a K3 surface and $H\in\mathrm{Pic}(S)$
 is a primitive polarization class with $H^2 = 2g-2$.
The dimension of $\mathcal F_g$ is $19$.

\subsection{Kodaira dimension of \texorpdfstring{$\mathcal F_g$}{Fg}}
There are many values of $g$ (although a finite number)
for which we have no information about the Kodaira dimension of $\mathcal F_g$.
We now recall the known results. 
The following theorem has been established by Mukai 
in the cases $g\leq 13$ and $g=16,17,18,20$, by Nuer in the cases $g=14,20$, and 
by Farkas and Verra in the cases $g=14,22$. See Tables~\ref{unirationalityFgSmall} and \ref{unirationalityFgLarge} 
for precise references; see also \cite[Section~7]{HoffSta2}.
\begin{theorem}
\label{unirationality theorem}
 The moduli space $\mathcal F_g$ is {unirational} for any $g\leq 22$,
 with the possible exceptions $g\in\{15,19,21\}$.
\end{theorem}
\begin{table}[htbp]
     \centering
     %\tabcolsep=1.5pt 
     %\begin{adjustbox}{width=\textwidth}
     \begin{tabular}{|l|cccc|ccccc|c|c|c|}
     \hline 
     genus & $2$ & $3$ & $4$ & $5$ & $6$ & $7$ & $8$ & $9$ & $10$ & $11$ & $12$ & $13$ \\
     \hline 
     \cellcolor{yellow!20} unirationality &  \multicolumn{4}{c}{\cellcolor{green!10} classical} &  \multicolumn{5}{|c|}{\cellcolor{green!10} \cite{Muk88}}  & \cellcolor{yellow!20} \cite{Mukg11} & \cellcolor{yellow!20} \cite{mukai-biregularclassification} & \cellcolor{yellow!20} \cite{Mukg13} \\
     \hline 
     \cellcolor{green!10} explicit unir. & \multicolumn{4}{c}{\cellcolor{green!10} } & \multicolumn{5}{|c|}{\cellcolor{green!10} } & \cellcolor{green!10} \begin{tabular}{c} \cellcolor{green!10} \cite{HoffSta2} \\ \cellcolor{green!10} Code~\ref{F11uniratinality} \end{tabular} & \cellcolor{green!10} \begin{tabular}{c} \cellcolor{green!10} \cite{KPEpi}, see  \\ \cellcolor{green!10} also \cite{HoffSta2} \end{tabular} & \\
     \hline 
      \end{tabular}
          %\end{adjustbox}
          \caption{Unirationality of $\mathcal F_g$ for $g\leq 13$}
          \label{unirationalityFgSmall}
     \end{table}
     \begin{table}[htbp]
          \centering
          %\tabcolsep=1.5pt 
          %\begin{adjustbox}{width=\textwidth}
          \begin{tabular}{|c|c|c|c|c|c|c|c|c|}
          \hline 
           $14$ & $15$ & $16$ & $17$ & $18$ & $19$ & $20$ & $21$ & $22$\\
          \hline 
           \cellcolor{yellow!20} \begin{tabular}{c} \cellcolor{yellow!20} \cite{Nuer} \\ \cellcolor{yellow!20} \cite{FV18} \end{tabular} & & \cellcolor{yellow!20} \cite{Mukg16} & 
           \cellcolor{yellow!20} \cite{Mukg17}
           & \cellcolor{yellow!20}  \cite{Mukg1820} & & \cellcolor{yellow!20} \begin{tabular}{c} \cellcolor{yellow!20} \cite{Mukg1820} \\ \cellcolor{yellow!20} \cite{Nuer} \end{tabular} & & \cellcolor{yellow!20} \cite{FarkasVerraC42} \\
           \hline 
           \cellcolor{green!10} \begin{tabular}{c} \cellcolor{green!10} \cite{RS3} \\ \cellcolor{green!10} Table~\ref{Table K3 associated with cubics} \end{tabular} & & & &  & & \cellcolor{green!10} \begin{tabular}{c} \cellcolor{green!10} \cite{RS3} \\ \cellcolor{green!10} Code~\ref{code for F20} \end{tabular} & & \cellcolor{green!10} \begin{tabular}{c} \cellcolor{green!10} \cite{RS3} \\ \cellcolor{green!10} Code~\ref{code for F22} \end{tabular}\\
          \hline 
           \end{tabular}
               %\end{adjustbox}
               \caption{Unirationality of $\mathcal F_g$ for $14\leq g\leq 22$}
               \label{unirationalityFgLarge}
\end{table}
On the other side we have the following.
\begin{theorem}[\cite{Gritsenko2007}] \hspace{1pt}
     \label{non unirazionalita di Fg}
\begin{itemize}
 \item The moduli space $\mathcal F_g$ 
 is of {general type} for any $g\geq 47$,
 with the possible exceptions $g\in\{48,49,50,52,53,54,56,57,60,62\}$.
 \item The Kodaira dimension $\kappa(\mathcal F_g)$ is {non-negative} for $g\geq 41$, 
 with the possible exceptions $g\in\{42,45,46,48\}$.
\end{itemize}
\end{theorem}

\subsection{Connections with cubic fourfolds and GM fourfolds}\label{Subsection Connections with cubic and GM fourfolds}
Let $d>6$ be $\mathcal C$-admissible and $[X]\in \mathcal C_d$ general. If
$d\equiv 2\,\mathrm{mod}\,6$ then $X$ admits a unique associated K3 surface,
while if $d\equiv 0\,\mathrm{mod}\,6$ then $X$ admits two associated K3 surfaces. Indeed, we have the following:
\begin{theorem}[\cite{Has00}]\label{HassettFgCd}
Assume $d>6$ is $\mathcal C$-admissible.
There is a dominant rational map 
\[\varphi_d : \mathcal F_{\frac{d}{2}+1}\dashrightarrow \mathcal C_d , \quad [S]\mapsto [X]: S \mbox{ is associated with }X\]
which is birational for $d\equiv 2\,\mathrm{mod}\,6$
and a degree $2$ cover for $d\equiv 0\,\mathrm{mod}\,6$.
\end{theorem}
The map $\varphi_{d}$ \emph{is not explicit} since it is defined at the level of Hodge structures. 
{The author is not able to determine an equation of $[X] = \varphi_d([S])$ from equations of $[S]$.}
\begin{remark}[\emph{Non-explicit} unirationality of $\mathcal F_{14}$ and $\mathcal F_{20}$]
     \label{rmk unirationality F14}
Nuer in \cite{Nuer} showed that $\mathcal C_{26}$ and $\mathcal C_{38}$
are unirational.  From this and the birationality of the maps $\varphi_{26}$ and $\varphi_{38}$,
he deduced the unirationality of $\mathcal F_{14}$ and $\mathcal F_{20}$.
\end{remark}
\begin{remark}[\emph{Non-explicit} unirationality of $\mathcal F_{14}$ and $\mathcal F_{22}$]
     \label{rmk unirationality F14 and F22}
We recall the results of Farkas and Verra on the unirationality of 
$\mathcal F_{14}$ and $\mathcal F_{22}$.
Let us consider 
\begin{align*}
 \mathfrak h_{\mathrm{scr}}^{3,7} &= \mathrm{PGL}(6)\mbox{-quotient of the Hilbert scheme of }3\mbox{-nodal septic scrolls} ,\\
 \mathfrak h_{\mathrm{scr}}^{8,9} &= \mathrm{PGL}(6)\mbox{-quotient of the Hilbert scheme of }8\mbox{-nodal nonic scrolls} ,
\end{align*}
and the incidence correspondences
\begin{align*}
\mathfrak X_{26} &= \left\{[R,X]: [R]\in\mathfrak h_{\mathrm{scr}}^{3,7},\ [X]\in |H^0(\mathcal I_{R / \PP^5}(3))|   \right\}/\mathrm{PGL}(6) ,\\
 \mathfrak X_{42} &= \left\{[R,X]: [R]\in\mathfrak h_{\mathrm{scr}}^{8,9},\ [X]\in |H^0(\mathcal I_{R / \PP^5}(3))|   \right\}/\mathrm{PGL}(6) .
\end{align*}
The main results in \cite{FV18} and \cite{FarkasVerraC42} state 
 that $\mathfrak X_{26}$ is {rational} and that $\mathfrak X_{42}$ is {unirational}, and moreover that
 we have two commutative diagrams:
  \begin{equation*}
\xymatrix{
 \mathcal F_{14,1} \ar@{-->}[rr]^{\simeq}_{\begin{subarray}{c} \mbox{\tiny{\quad \emph{non-explicit map}}} \end{subarray}} \ar[d] && \mathfrak X_{26} \ar[d] \\
 \mathcal F_{14} \ar@{-->}[rr]^{\varphi_{26}}_{\begin{subarray}{c} \mbox{\tiny{\quad \emph{non-explicit map}}} \end{subarray}} && \mathcal C_{26}
 }\quad \quad \quad 
 \xymatrix{
 \mathcal F_{22,1} \ar@{-->}[rr]^{\simeq}_{\begin{subarray}{c} \mbox{\tiny{\quad \emph{non-explicit map}}} \end{subarray}} \ar[d] && \mathfrak X_{42} \ar[d] \\
 \mathcal F_{22} \ar@{-->}[rr]^{\varphi_{42}}_{\begin{subarray}{c} \mbox{\tiny{\quad \emph{non-explicit map}}} \end{subarray}} && \mathcal C_{42}
 }
 \end{equation*}
From this the unirationality of 
$\mathcal F_{14}$ and $\mathcal F_{22}$ follows.
\end{remark}
The following remark 
is particularly useful for us. It
follows from 
Theorem~\ref{HassettFgCd} in the case of cubic fourfolds,
and
from the main result of \cite{BrakkePertusi} in the case of GM fourfolds.
\begin{remark}[\cite{Has00,BrakkePertusi}]\label{rmk BrakkePertusi}
Let $d$ be a $\mathcal{C}$-admissible (resp., $\mathcal{GM}$-admissible) value. 
Let $X$ be a cubic (resp., GM) fourfold corresponding 
to a general point in an irreducible component of the Noether--Lefschetz locus of 
cubic (resp., GM) fourfolds of discriminant $d$.
Let $S$ be a K3 surface of genus $g=\frac{d}{2}+1$ which is associated with the fourfold $X$.
Then $S$ corresponds to a general point in the moduli space $\mathcal{F}_{g}$.
\end{remark}

\subsection{Explicit unirationality of \texorpdfstring{$\mathcal F_g$}{Fg}}
We are interested in the \emph{explicit unirationality} of the moduli space $\mathcal{F}_g$,
that is, in finding
 an explicit dominant rational map from a  projective space to $\mathcal{F}_g$. More precisely, with ``\emph{explicit}'' we mean that there is 
 a computer-implementable procedure
 to 
determine the equations of the general member of $\mathcal{F}_g$
as a function of a 
number of specific independent variables.
The original methods used in the proof of Theorem~\ref{unirationality theorem}
provide
the explicit unirationality of $\mathcal F_g$ only
for $g\leq 12$ and $g\neq 11$ (see also \cite{KPEpi} for the case $g=12$).

In the following, we recall a construction from \cite{RS3} and \cite{HoffSta2}
on the explicit unirationality of $\F_g$ for some values of $g$. Then we show
how this can be executed in practice to get the equations 
of general K3 surfaces.
To be more precise we have the following (see also Tables~\ref{unirationalityFgSmall} and \ref{unirationalityFgLarge}).
\begin{theorem}[\cite{RS3,HoffSta2}]\label{teo esplicity unir}
 The moduli space $\mathcal F_g$ is explicitly unirational for $g=6,8,11,14,20,22$.
\end{theorem}
Note that  cases $g=6$ and $g=8$ are very elementary: general K3 surfaces
of genus $6$ and $8$ can be respectively realized 
as linear sections of quadratic sections of $\mathbb{G}(1,4)\subset\PP^9$
and as linear sections of $\mathbb{G}(1,5)\subset\PP^{14}$.
However 
the new method works in the same way in all cases.
 The main ingredient used 
 is 
 a \emph{good description}
 of a 
 {unirational component} 
 of the 
 {Noether--Lefschetz locus} 
 in the moduli space of cubic fourfolds or GM fourfolds
 such that the fourfolds in this component 
 have a 
 Hodge-associated K3 surface of genus $g$. 
 So,
 taking into account Table~\ref{TableAdmissible} and Theorems~\ref{non unirazionalita di Cd} and \ref{non unirazionalita di Fg},
 we do not exclude that the same method 
 might be also applied to other values of $g$ in the set
 \[\{6, 8, 11, 14, \colorablu{18}, 20, 22, \colorablu{26}, \colorablu{27}, \colorablu{30}, \colorablu{32}, \colorablu{35}, \colorablu{38}, \colorablu{40}\} . \]
In the next subsection 
we
  provide more details on the procedure,
which is strongly connected with the proof of Theorem~\ref{solutionKuz}.

\subsection{An overview of the construction to obtain Theorem~\ref{teo esplicity unir}}
For more details on this subsection, we refer to \cite[Section~4]{RS3} and \cite[Section~6]{HoffSta2}.
 Let $\mathbb{V}$ be either $\PP^5$ or a Del Pezzo fivefold 
 $\YY$, and 
let $r$ be respectively equal to  $3$ and $2$, 
the degree of the hypersurfaces we consider in $\mathbb{V}$.
For each of the components $\mathcal{D}$
considered in 
Theorem~\ref{solutionKuz} parameterizing fourfolds 
of discriminant $d$ (with the exception of $\mathcal{GM}_{10}^{''}$),
we are able to exhibit an explicit unirational family 
$\mathcal S\subset \mathrm{Hilb}(\mathbb{V})$  of surfaces in $\mathbb{V}$
such that the following hold (actually, we have more than one example; 
see Tables~\ref{Table K3 associated with cubics} and \ref{Table K3 associated with GMs}):
\begin{enumerate}
\item The closure of the locus of smooth hypersurfaces $X\subset\mathbb{V}$ of degree $r$ containing some surface 
$S$ of the family $\mathcal{S}$ describes a hypersurface 
in $\PP(H^0(\mathcal{O}_{\mathbb{V}}(r)))$,
which gives rise to the component $\mathcal{D}$ of 
the Noether--Lefschetz locus in the corresponding moduli space $\mathcal C$ or $\mathcal{GM}$.
\item For some $e\geq1$, the general member $[S]\in\mathcal S$ 
admits a \emph{congruence of 
$(re-1)$-secant curves of degree $e$},
parametrized by a variety $\mathcal H$.
Thus we have a universal family 
\[\mathcal{X} = \{([C],p):[C]\in\mathcal H,\ p\in C\subset\mathbb{V}\}\subset \mathcal H\times \mathbb{V}\]
and 
two projections
\begin{equation}\label{proiezioni}
\xymatrix{
& \mathcal{X}\ar[dr]^{\pi}\ar[dl]&\\
\mathcal{H} & & \mathbb{V},}
\end{equation}
where $\pi$ is birational and its general fiber 
$\pi^{-1}(p)$ corresponds to the unique curve $[C]\in\mathcal H$ 
passing through $p$ and cutting $S$ at $re-1$ points.
\item 
 The linear system 
$|H^0(\mathcal I_{S/\mathbb{V}}^e(re-1)|$ of hypersurfaces in $\mathbb{V}$
of degree $re-1$ with points of multiplicity $e$ along $S$
defines a dominant rational map 
\[\mu:\mathbb{V}\dashrightarrow W \] 
onto a prime Fano fourfold $W$ of index $s=i(W)$. 
In particular, {we can realize the general  curve of the congruence 
as the general fiber 
of the map $\mu$}.
\item 
By degree reasons,
the
restriction of $\mu$ to 
a general hypersurface
$[X]\in|H^0(\mathcal I_{S/\mathbb{V}}(r))|$ of degree $r$ that contains $S$ 
gives a birational map 
$\mu|_X:X\dashrightarrow W $.
We have that the inverse map of $\mu|_X$ is defined by the linear system 
of hypersurfaces in $W$ of degree $se-1$  having points 
of multiplicity $e$ along an irreducible surface $U\subset W$.
These maps fit in the following commutative diagram:
\begin{equation}
\label{diagramma}
\xymatrix{&& Z\cap \mathbb{P}^{N-1} & \\ \\
& {\mathbb{V}} \ar@/^1.9pc/@{-->}[rrd]^{\mu} &&&  \\
S 
\ar@{^{(}->}[r] 
& X\ar@{^{(}->}[u]\ar@/^1.0pc/@{-->}[rr]^{|H^0(\mathcal{I}^{e}_{S}(re-1))|} 
\ar@/^3.0pc/@{-->}[uuur]_{|H^0(\mathcal{I}_{S}(r))|}^{\varphi|_X}
&& W \ar@/^1.0pc/@{-->}[ll]^{|H^0(\mathcal{I}^{e}_{U}(s e-1))|}
\ar@/_3.0pc/@{-->}[uuul]^{|H^0(\mathcal{I}_{U}(s))|} & U \ar@{_{(}->}[l]
} 
\end{equation}
where $Z\subset\PP^{N}$
is the closure of the image of the map 
$\varphi:\mathbb{V}\dashrightarrow\PP^{N}$
defined by the linear system of hypersurfaces of degree $r$ through $S$.
\item Finally, the surface $U\subset W$ can be obtained 
as a linear projection of a 
 minimal K3 surface $\tilde{U}\subset\PP^{\frac{d}{2}+1}$
of degree $d$ and genus $\frac{d}{2}+1$, 
which is associated with the fourfold $X$.
\end{enumerate}
In this framework, we define
the \emph{fundamental locus} $E\subset\mathbb{V}$
 of the congruence as the base locus 
 of the inverse of the projection $\pi$ in \eqref{proiezioni}, that is, 
$E={\{p\in\mathbb{V} : \dim(\pi^{-1}(p))>0\}}$.
Then notice that  
there are two kinds of points on $U$:
\begin{itemize}
 \item the points $p\in W$ such that $\overline{\mu^{-1}(p)}$ is a curve of the congruence 
 but contained in the hypersurface $X\subset\mathbb{V}$;
 \item and the points $p\in\overline{\mu(E)}\subset W$
 in the image via $\mu$ of $E$; 
 so that  $\dim \overline{\mu^{-1}(p)} \geq 2$.
\end{itemize}
A key remark here is that
 the points of the second kind do not depend on the hypersurface $X\subset\mathbb{V}$ through $S$.
 So that if $X'\subset\mathbb{V}$ is another general hypersurface of degree $r$ through $S$,
 and $U'\subset W$ is the 
 corresponding surface defining the inverse of the restriction of $\mu$ to $X'$, we must have
 $ \overline{\mu(E)}\subseteq U\cap U' $.
 Actually, in all cases we have analyzed, it turns out that 
 $\overline{\mu(E)}$ is the union 
 of the exceptional curves on $U$ and on $U'$, which are recoverable as the
 top-dimensional components of $U\cap U'$.
 In particular, we 
 can determine their equations   from the equations of $S$ in $\mathbb{V}$.
 Once we have computed the surface $U$ and its exceptional curves $C_1,\ldots,C_t$, we get the minimal associated K3 surface $\tilde{U}\subset\mathbb{P}^{\frac{d}{2}+1}$ as the closure of the image of the rational map $U\dashrightarrow \mathbb{P}^{\frac{d}{2}+1}$ defined by 
 $|H+\sum_{i=1}^{t}\deg(C_i) C_i|$ (in some cases a normalization of $U$ is also needed).
 We summarize with the following diagram:
\begin{equation}
\label{diagramma2}
\xymatrix{
& X' \ar@{_{(}->}[dr] & & & & U' \ar@{_{(}->}[ld] \ar@{-->}[rr]^{|H+\sum_{i=1}^{t}\deg(C_i) C_i|} & & \tilde{U'}\\
S \ar@{_{(}->}[ur] \ar@{^{(}->}[rd]  &&  \mathbb{V} \ar@{-->}[rr]^{\mu} && W \ar@/_1.0pc/@{-->}[lllu]_{|H^0(\mathcal{I}^{e}_{U'}(s e-1))|} \ar@/^1.0pc/@{-->}[llld]^{|H^0(\mathcal{I}^{e}_{U}(s e-1))|} & C_1,\ldots,C_t \ar@{_{(}->}[l]\\
& X \ar@{^{(}->}[ur] & & & & U \ar@{^{(}->}[lu] \ar@{-->}[rr]_{|H+\sum_{i=1}^{t}\deg(C_i) C_i|} & & \tilde{U}
}
\end{equation}
 By Remark~\ref{rmk BrakkePertusi} we deduce that 
 if the pair $(S,X)$ 
 is general enough in 
  the incidence correspondent 
 ${\mathcal X}_{\mathcal S} = \left\{(S,X): [S]\in \mathcal S,\ [X]\in |H^0(\mathcal I_{S/\mathbb{V}}(r))|   \right\}  $,
 then 
 the K3 surface $\tilde{U}$
 will be general enough 
 in the moduli space $\mathcal F_{g}$ with $g=\frac{d}{2}+1$. 
 If the coordinates of the generic point of ${\mathcal X}_{\mathcal S}$ 
 belong to 
 a pure transcendental extension of $\mathbb C$ (which 
 is always possible if $\mathcal S$ is unirational),
 then the same holds for the 
 coordinates of the generic point of $\mathcal F_g$, thus obtaining the unirationality of $\mathcal F_g$.

\subsection{Implementation of the construction}
Everything said in the previous subsection
is automatically detected and handled by the function 
\begin{tcolorbox}[breakable=true,boxrule=0.5pt,opacityback=0.1,enhanced jigsaw]  
\begin{center}
     \href{https://faculty.math.illinois.edu/Macaulay2/doc/Macaulay2/share/doc/Macaulay2/SpecialFanoFourfolds/html/_associated__K3surface.html}{\texttt{associatedK3surface}} 
 \end{center}
\end{tcolorbox}
\noindent 
This function 
takes as input a Hodge-special fourfold 
represented by a pair $(S,X)$, with $S\subset X\subset \mathbb{V}$ as described above.
Then, if nothing goes wrong, it returns four objects:\footnote{This function is under development and subject to change. 
In future versions the output could be encapsulated in a single object.}
\begin{itemize}
\item the dominant rational map $\mu:\mathbb{V} \dashrightarrow W$ defined by the linear system of hypersurfaces of degree $re-1$ having points of multiplicity $e$ along $S$;    
\item the surface $U\subset W$ determining the inverse map of the restriction of $\mu$ to $X$;              
\item the list $C=\{C_1,C_2,\ldots\}$ of the exceptional curves on the surface $U$ (more precisely, $C_1$ is the union of the exceptional lines, 
$C_2$ of the conics, and so on);         
\item the birational map $f:U\dashrightarrow \tilde{U}\subset \mathbb{P}^{\frac{d}{2}+1}$ defined by 
$|H+\sum_{i=1}^{t}\deg(C_i) C_i|$.     
\end{itemize}
Therefore, by taking the image of the last map  we get an 
associated K3 surface $\tilde{U}\subset\mathbb{P}^{\frac{d}{2}+1}$ with the fourfold $X$. 

Here are a few more details on the 
computation behind the scenes.
The first task of the procedure is to try to detect the congruence of curves. 
This is done by considering the map $\varphi:\mathbb{V}\dashrightarrow Z\subset\PP^{N}$ 
that appears in \eqref{diagramma}
defined by the linear system 
of hypersurfaces of degree $r$ through $S$, checking that it is birational, and then
 analyzing the lines on $Z$ passing through its general point. Indeed, if $\varphi$ is birational,
 it  induces a 1--1 correspondence:
 \[
     \bigcup_{e\geq 1} \left\{{\begin{minipage}[c]{0.4\textwidth} \normalsize{curves of degree $e$ in $\mathbb{V}$ passing though a general  point $p\in\mathbb{V}$ and that  are  $(re-1)$-secant to $S$}\end{minipage}} \right\}   
     \stackrel{\simeq}{\longrightarrow} 
     \left\{{\begin{minipage}[c]{0.3\textwidth} \normalsize{lines contained in $Z\subset\PP^N$ 
         and  passing through $ \varphi(p) $}\end{minipage}} \right\} .
 \]
Once the congruence is detected, the next tasks are to compute 
the map $\mu:\mathbb{V}\dashrightarrow W$ and the surface $U\subset W$ corresponding to $X$.
There are several ways to determine the equations of $U$.
The strategy often adopted by \emph{Macaulay2}
 is based on the fact 
that 
we have a 1--1 correspondence between the two linear systems
$|H^0(\mathcal I_{S/X}(r))|$ and $|H^0(\mathcal I_{U/W}(s))|$,
which follows from the commutativity of the diagram~\eqref{diagramma}.
In most of the known cases, $S$ is cut out by hypersurfaces 
of degree $r$ and $U$ is cut out by hypersurfaces of degree $s$.
\subsection{Running the construction with a example}\label{quarta riga della tabella}
In the following,
we run the procedure with
an example of fourfold constructed as described 
in the last row of Table~\ref{Table K3 associated with GMs}.
This is a very special GM fourfold $X\subset \YY= \GG(1,4)\cap \PP^8$ which contains 
a smooth rational surface $S\subset\PP^8$ of degree $9$ and sectional genus $2$ 
with class $5\sigma_{3,1}+4\sigma_{2,2}$ in $\GG(1,4)$. Thus $[X]\in \mathcal{GM}_{26}^{''}$, by \eqref{autointersezioneGM} and \eqref{discriminanteGM}, although not all fourfolds in $\mathcal{GM}_{26}^{''}$ are of this type (see Code~Example~\ref{codimension 2 countParametri}).
It seems relevant to notice that 
this surface $S$, as a subvariety of $\PP^8$, is the same surface that appears in 
the fourth row of Table~\ref{Table K3 associated with GMs} and in 
Code~Example~\ref{F11uniratinality}, that is the surface considered in 
\cite{HoffSta} to describe the component $\mathcal{GM}_{20}$
(indeed, both can be realized as the image of the plane 
via the linear system of quartic curves through three simple
points and one double point in general position).
This surface also admits a congruence of $3$-secant conics
inside $\YY$, and the linear system of cubic hypersurfaces 
in $\YY$ with double points along $S$ gives a dominant rational map 
$\mu:\YY\dashrightarrow W\subset\PP^5$ 
onto a smooth quadric hypersurface $W\subset \PP^5$.
The restriction of $\mu$ to $X$ is a birational map $X\dashrightarrow W$ 
whose inverse is 
defined by the linear system of 
hypersurfaces of degree $7$ in $W$ with 
double points along a smooth irreducible surface $U\subset W$ 
of degree $13$ and sectional genus $11$.
This surface $U$ has four exceptional lines $L_1,\ldots,L_4$
 and one exceptional twisted cubic $C$. The linear system $|H + L_1 + \cdots + L_4 + 3 C|$ on $U$
  gives a birational map onto a minimal K3 surface $\tilde{U}\subset\PP^{14}$ of genus $14$. 
 As mentioned earlier, 
 everything is done by \emph{Macaulay2}
 as in the following.
\begin{codeexample}
     We construct a fourfold $X$ as described 
     in the last row of Table~\ref{Table K3 associated with GMs}. Then we call the function  \texttt{associatedK3surface}
     (by turning on the ``\texttt{Verbose}'' option).
\label{codePerQuartoEsempioInTabella}
\begin{tcolorbox}[breakable=true,boxrule=0.5pt,opacityback=0.1,enhanced jigsaw]  
{\footnotesize
\begin{Verbatim}[commandchars=&!$]
&colore!darkorange$!i11 :$ &colore!darkorchid$!time$ X = &colore!bleudefrance$!specialGushelMukaiFourfold$([4,5,1],[2,3,0]);
&colore!darkorchid$!      -- used 9.5484 seconds$
&colore!circOut$!o11 :$ &colore!output$!ProjectiveVariety, GM fourfold containing a surface$
&colore!output$!      of degree 9 and sectional genus 2$
&colore!darkorange$!i12 :$ &colore!bleudefrance$!describe$ X
&colore!circOut$!o12 =$ &colore!output$!Special Gushel-Mukai fourfold of discriminant 26('')$
&colore!output$!      containing a smooth surface in PP^8 of degree 9 and$
&colore!output$!      sectional genus 2 cut out by 19 hypersurfaces of degree 2$
&colore!output$!      and with class in G(1,4) given by 5*s_(3,1)+4*s_(2,2)$
&colore!output$!      Type: ordinary$
&colore!darkorange$!i13 :$ &colore!darkorchid$!time$ (mu,U,C,f) = &colore!bleudefrance$!associatedK3surface$(X,&colore!bleudefrance$!Verbose$=>&colore!airforceblue$!true$);
&colore!output$!      -- computing the Fano map mu from the fivefold in PP^8 to PP^5 defined$
&colore!output$!         by the hypersurfaces of degree 3 with points of multiplicity 2$
&colore!output$!         along the surface S of degree 9 and genus 2$
&colore!output$!      -- computing the surface U corresponding to the fourfold X$
&colore!output$!      -- computing the surface U' corresponding to another fourfold X'$
&colore!output$!      -- computing the 4 exceptional line(s) in U ∩ U'$
&colore!output$!      -- computing the top components of U ∩ U'$
&colore!output$!      -- computing the map f from U to the minimal K3 surface of degree 26$
&colore!output$!      -- computing the image of f$
&colore!darkorchid$!      -- used 58.1472 seconds$
&colore!darkorange$!i14 :$ f;
&colore!circOut$!o14 :$ &colore!output$!RationalMap (rational map from U to PP^14)$
&colore!darkorange$!i15 :$ &colore!bleudefrance$!? ideal image$ f
&colore!circOut$!o15 =$ &colore!output$!surface of degree 26 and sectional genus 14 in PP^14$
&colore!output$!      cut out by 66 hypersurfaces of degree 2$
\end{Verbatim} 
 } \noindent
\end{tcolorbox}
\end{codeexample}   
The congruence of $3$-secant conics to $S$ can be detected by
 considering the map 
$\varphi:\YY\dashrightarrow Z\subset\PP^{13}$ 
defined by the quadrics through $S$. Indeed,
we have that 
through the general point of $Z$ there pass $6$ lines, which come 
 from five $1$-secant lines to $S$ 
 and one single $3$-secant conic to $S$. This is performed by the following code.
 \begin{codeexample}
\label{codePerDetectCongruence}
     We run the function \href{https://faculty.math.illinois.edu/Macaulay2/doc/Macaulay2/share/doc/Macaulay2/SpecialFanoFourfolds/html/_detect__Congruence.html}{\texttt{detectCongruence}} 
     for the fourfold $X$ constructed in Code Example~\ref{codePerQuartoEsempioInTabella} (last row of 
     Table~\ref{Table K3 associated with GMs}).  
\begin{tcolorbox}[breakable=true,boxrule=0.5pt,opacityback=0.1,enhanced jigsaw]          
{\footnotesize
\begin{Verbatim}[commandchars=&!$]
&colore!darkorange$!i16 :$ &colore!darkorchid$!time$ c = &colore!bleudefrance$!detectCongruence$(X,&colore!bleudefrance$!Verbose$=>&colore!airforceblue$!true$);
&colore!output$!      number lines contained in the image and passing through a general point: 6$
&colore!output$!      number 1-secant lines = 5$
&colore!output$!      number 3-secant conics = 1$
&colore!darkorchid$!      -- used 6.18911 seconds$
&colore!circOut$!o16 :$ &colore!output$!Congruence of 3-secant conics to surface in a Del Pezzo fivefold$
\end{Verbatim} 
 } \noindent
\end{tcolorbox}
\noindent 
 At this point, \emph{Macaulay2} knows that the fourfold $X$ is rational.
  A command such as 
  \href{https://faculty.math.illinois.edu/Macaulay2/doc/Macaulay2/share/doc/Macaulay2/MultiprojectiveVarieties/html/_parametrize_lp__Multiprojective__Variety_rp.html}{\texttt{parametrize(X)}}  gives us a rational parameterization of $X$.
\end{codeexample}
We conclude with a count of parameters (see Proposition~\ref{Prop parameterCount}) 
which shows that the family of fourfolds $X$ 
as in the last row of Table~\ref{Table K3 associated with GMs} describes a locus of codimension $2$ 
in $\mathcal{GM}$, hence of codimension $1$ in  $\mathcal{GM}_{26}^{''}$.
\begin{codeexample}
     \label{codimension 2 countParametri} 
     We run the function 
     \href{https://faculty.math.illinois.edu/Macaulay2/doc/Macaulay2/share/doc/Macaulay2/SpecialFanoFourfolds/html/_parameter__Count.html}{\texttt{parameterCount}}
     for the fourfold $X$ constructed in Code Example~\ref{codePerQuartoEsempioInTabella} (last row of 
     Table~\ref{Table K3 associated with GMs}).    
\begin{tcolorbox}[breakable=true,boxrule=0.5pt,opacityback=0.1,enhanced jigsaw]   
{\footnotesize
\begin{Verbatim}[commandchars=&!$]
&colore!darkorange$!i17 :$ &colore!darkorchid$!time$ &colore!bleudefrance$!parameterCount$(X,&colore!bleudefrance$!Verbose$=>&colore!airforceblue$!true$)
&colore!output$!     -- h^1(N_{S,Y}) = 0$
&colore!output$!     -- h^0(N_{S,Y}) = 24$
&colore!output$!     -- h^1(O_S(2)) = 0, and h^0(I_{S,Y}(2)) = 14 = h^0(O_Y(2)) - \chi(O_S(2));$
&colore!output$!     -- in particular, h^0(I_{S,Y}(2)) is minimal$
&colore!output$!     -- h^0(N_{S,X}) = 0$
&colore!output$!     -- codim{[X] : S ⊂ X ⊂ Y} <= 2$
     &colore!darkorchid$!-- used 337.969 seconds$
&colore!circOut$!o17 =$ &colore!output$!(2, (14, 24, 0))$
\end{Verbatim}
} \noindent 
\end{tcolorbox}
\end{codeexample}

\subsection{Running the construction with other examples}

\begin{codeexample}[Explicit unirationality of $\mathcal F_{11}$]
     \label{F11uniratinality}
     Here
     we construct a random GM fourfold $X\subset\YY=\GG(1,4)\cap\PP^8$ in the component 
     $\mathcal{GM}_{20}$ which contains a smooth 
     rational surface $S\subset\YY$
     of degree $9$, sectional genus $2$, and class $6\sigma_{3,1}+3\sigma_{2,2}$ in $\GG(1,4)$,
     as described in 
     \cite{HoffSta} and \cite[Remark~5.8]{RS3} 
     (see also the fourth 
     row of Table~\ref{Table K3 associated with GMs}).
     Therefore, the corresponding associated K3 surface is a random K3 surface of genus $11$ 
     in $\PP^{11}$. 
     In the notation of the diagrams \eqref{diagramma} and \eqref{diagramma2}, 
     we have $r=2$, $e=2$, $s=5$,
     $W=\PP^4$,
     $U\subset W$ is a non-normal surface  of degree $10$ and sectional genus $8$,
     and the exceptional curves on (the normalization of) $U$ are one line and one twisted cubic curve.
     See also \cite[Section~6]{HoffSta2} for more details on this calculation.
\begin{tcolorbox}[breakable=true,boxrule=0.5pt,opacityback=0.1,enhanced jigsaw]  
{\footnotesize
\begin{Verbatim}[commandchars=&!$]
&colore!darkorange$!i18 :$ X = &colore!bleudefrance$!specialGushelMukaiFourfold$ "general GM 4-fold of discriminant 20";
&colore!circOut$!o18 :$ &colore!output$!ProjectiveVariety, GM fourfold containing a surface$ 
&colore!output$!      of degree 9 and sectional genus 2$
&colore!darkorange$!i19 :$ &colore!bleudefrance$!describe$ X
&colore!circOut$!o19 =$ &colore!output$!Special Gushel-Mukai fourfold of discriminant 20$
&colore!output$!      containing a smooth surface in PP^8 of degree 9 and$
&colore!output$!      sectional genus 2 cut out by 19 hypersurfaces of degree 2$
&colore!output$!      and with class in G(1,4) given by 6*s_(3,1)+3*s_(2,2)$
&colore!darkorange$!i20 :$ &colore!darkorchid$!time$ (mu,U,C,f) = &colore!bleudefrance$!associatedK3surface$(X,&colore!bleudefrance$!Verbose$=>&colore!airforceblue$!true$);
      &colore!output$!-- computing the Fano map mu from the fivefold in PP^8 to PP^4 defined$
      &colore!output$!   by the hypersurfaces of degree 3 with points of multiplicity 2$
      &colore!output$!   along the surface S of degree 9 and genus 2$
      &colore!output$!-- computing the surface U corresponding to the fourfold X$
      &colore!output$!-- computing the 1 exceptional line(s) in U ∩ U'$
      &colore!output$!-- computing the top components of U ∩ U'$
      &colore!output$!-- computing desingularization of U$
      &colore!output$!-- computing the map f from U to the minimal K3 surface of degree 20$
      &colore!output$!-- computing the image of f$
 &colore!darkorchid$!     -- used 114.893 seconds$
&colore!darkorange$!i21 :$ f;
&colore!circOut$!o21 :$ &colore!output$!RationalMap (rational map from U to PP^11)$
&colore!darkorange$!i22 :$ &colore!bleudefrance$!? ideal image$ f
&colore!circOut$!o22 =$ &colore!output$!surface of degree 20 and sectional genus 11 in PP^11$
&colore!output$!      cut out by 36 hypersurfaces of degree 2$
\end{Verbatim} 
 } \noindent
\end{tcolorbox}
\end{codeexample}
\begin{remark}
 The command 
 \href{https://faculty.math.illinois.edu/Macaulay2/doc/Macaulay2/share/doc/Macaulay2/K3Surfaces/html/___K3_lp__Z__Z_rp.html}{\texttt{K3(11)}}
  provided by the package 
 \href{https://faculty.math.illinois.edu/Macaulay2/doc/Macaulay2/share/doc/Macaulay2/K3Surfaces/html/index.html}{\emph{K3Surfaces}} \cite{K3SurfacesSource}
 executes the procedure in Code~Example~\ref{F11uniratinality} and returns a random K3 surface of genus $11$. The same command can be called with other values of $g$,
 for which an explicit unirationality construction is known (see Tables~\ref{unirationalityFgSmall} and \ref{unirationalityFgLarge}).
\end{remark}
 
\begin{codeexample}[Explicit unirationality of $\mathcal F_{20}$]
     \label{code for F20}
     Here
     we construct a random cubic fourfold  in the component 
     $\mathcal{C}_{38}$ which contains a smooth 
     rational surface 
     of degree $10$ and sectional genus $6$ as described in 
     \cite{Nuer} and \cite[Section~3]{RS3} (see also the fifth 
     row of Table~\ref{Table K3 associated with cubics}).
     Therefore, the corresponding associated K3 surface is a random K3 surface of genus $20$ 
     in $\PP^{20}$. 
     In the notation of the diagrams \eqref{diagramma} and \eqref{diagramma2}, 
     we have $r=3$, $e=2$, $s=5$,
     $W=\PP^4$,
     $U\subset W$ is a smooth surface of degree $12$ and sectional genus $14$,
     and the exceptional curves on $U$ are $10$ lines and one rational normal quartic curve.
     See also \cite[Section~6.1]{RS3} for more details on this calculation.
\begin{tcolorbox}[breakable=true,boxrule=0.5pt,opacityback=0.1,enhanced jigsaw]   
{\footnotesize 
\begin{Verbatim}[commandchars=&!$]
&colore!darkorange$!i23 :$ S = &colore!bleudefrance$!surface$ {10,0,0,10};
&colore!circOut$!o23 :$ &colore!output$!ProjectiveVariety, surface in PP^5$
&colore!darkorange$!i24 :$ X = &colore!bleudefrance$!specialCubicFourfold$ S;
&colore!circOut$!o24 :$ &colore!output$!ProjectiveVariety, cubic fourfold containing a surface$
&colore!output$!      of degree 10 and sectional genus 6$
&colore!darkorange$!i25 :$ &colore!bleudefrance$!describe$ X
&colore!circOut$!o25 =$ &colore!output$!Special cubic fourfold of discriminant 38$
&colore!output$!      containing a smooth surface of degree 10 and sectional genus 6$
&colore!output$!      cut out by 10 hypersurfaces of degree 3$
&colore!darkorange$!i26 :$ &colore!darkorchid$!time$ (mu,U,C,f) = &colore!bleudefrance$!associatedK3surface$(X,&colore!bleudefrance$!Verbose$=>&colore!airforceblue$!true$);
      &colore!output$!-- computing the Fano map mu from PP^5 to PP^4 defined by the$
      &colore!output$!   hypersurfaces of degree 5 with points of multiplicity 2$ 
      &colore!output$!   along the surface S of degree 10 and genus 6$
      &colore!output$!-- computing the surface U corresponding to the fourfold X$
      &colore!output$!-- computing the surface U' corresponding to another fourfold X'$
      &colore!output$!-- computing the 10 exceptional line(s) in U ∩ U'$
      &colore!output$!-- computing the top components of U ∩ U'$
      &colore!output$!-- computing the map f from U to the minimal K3 surface of degree 38$
      &colore!output$!-- computing the image of f$
&colore!darkorchid$!      -- used 173.767 seconds$
&colore!darkorange$!i27 :$ f;
&colore!circOut$!o27 :$ &colore!output$!RationalMap (rational map from U to PP^20)$
&colore!darkorange$!i28 :$ &colore!bleudefrance$!? ideal image$ f
&colore!circOut$!o28 =$ &colore!output$!surface of degree 38 and sectional genus 20 in PP^20$
&colore!output$!      cut out by 153 hypersurfaces of degree 2$
\end{Verbatim}
}
\end{tcolorbox}
\end{codeexample}

\begin{codeexample}[Explicit unirationality of $\mathcal F_{22}$]
     \label{code for F22}
     Here
     we construct a random cubic fourfold  in the component 
     $\mathcal{C}_{42}$ which contains a $5$-nodal  
     rational surface 
     of degree $9$ and sectional genus $2$ as described in 
     \cite[Section~4]{RS3} (see also the last row of Table~\ref{Table K3 associated with cubics}).
     Therefore, the corresponding associated K3 surface is a 
     random K3 surface of genus $22$ 
     in $\PP^{22}$. 
     In the notation of the diagrams \eqref{diagramma} and \eqref{diagramma2}, 
     we have $r=3$, $e=3$, $s=3$,
     $W=\mathbb{G}(1,4)\cap\PP^7$ is a Del Pezzo fourfold,
     $U\subset W$ is a smooth surface of degree $21$ and sectional genus $18$,
     and the exceptional curves on $U$ are $5$ lines and $4$ conics.
     See also \cite[Section~6.2]{RS3} 
     for more details on this calculation.
\begin{tcolorbox}[breakable=true,boxrule=0.5pt,opacityback=0.1,enhanced jigsaw]    
     {\footnotesize 
\begin{Verbatim}[commandchars=&!$]
&colore!darkorange$!i29 :$ X = &colore!bleudefrance$!specialCubicFourfold$ "general cubic 4-fold of discriminant 42";
&colore!circOut$!o29 :$ &colore!output$!ProjectiveVariety, cubic fourfold containing a surface$
&colore!output$!      of degree 9 and sectional genus 2$
&colore!darkorange$!i30 :$ &colore!bleudefrance$!describe$ X
&colore!circOut$!o30 =$ &colore!output$!Special cubic fourfold of discriminant 42$
&colore!output$!      containing a 5-nodal surface of degree 9 and sectional$
&colore!output$!      genus 2 cut out by 9 hypersurfaces of degree 3$
&colore!darkorange$!i31 :$ &colore!darkorchid$!time$ (mu,U,C,f) = &colore!bleudefrance$!associatedK3surface$(X,&colore!bleudefrance$!Verbose$=>&colore!airforceblue$!true$);
      &colore!output$!-- computing the Fano map mu from PP^5 to PP^7 defined$
      &colore!output$!   by the hypersurfaces of degree 8 with points of multiplicity 3$
      &colore!output$!   along the surface S of degree 9 and genus 2$
      &colore!output$!-- computing the surface U corresponding to the fourfold X$
      &colore!output$!-- computing the surface U' corresponding to another fourfold X'$
      &colore!output$!-- computing the 5 exceptional line(s) in U ∩ U'$
      &colore!output$!-- computing the top components of U ∩ U'$
      &colore!output$!-- computing the map f from U to the minimal K3 surface of degree 42$
      &colore!output$!-- computing the image of f$
&colore!darkorchid$!      -- used 692.743 seconds$
&colore!darkorange$!i32 :$ f;
&colore!circOut$!o32 :$ &colore!output$!RationalMap (rational map from U to PP^22)$
&colore!darkorange$!i33 :$ &colore!bleudefrance$!? ideal image$ f
&colore!circOut$!o33 =$ &colore!output$!surface of degree 42 and sectional genus 22 in PP^22$
&colore!output$!      cut out by 190 hypersurfaces of degree 2$
\end{Verbatim}
}
\end{tcolorbox}
\end{codeexample}  

\clearpage

\section{Summary tables of examples}

\begin{table}[htbp]
     \renewcommand{\arraystretch}{1.300}
     \centering
     \tabcolsep=1.0pt 
     \footnotesize
     %\begin{adjustbox}{width=\textwidth}
     \begin{tabular}{|c|c|c|c|c|c|c|c|c|}
     \hline
     \rowcolor{gray!5.0}
       {\begin{tabular}{c} Surface $S$ in \\ $\YY=\mathbb{G}(1,4)\cap\PP^8$ \end{tabular}} & $K_S^2$ & {\begin{tabular}{c} Class in \\ $\mathbb{G}(1,4)$\end{tabular}} & {\begin{tabular}{c} Codim \\ in $\mathcal{GM}$ \end{tabular}} & \begin{tabular}{c} Locus \\ in  $\mathcal{GM}$ \end{tabular} & $h^0(\mathcal I_{S/\YY}(2))$ & $h^0(N_{S/\YY})$ & $h^0(N_{S/X})$ &
      {\begin{minipage}[c]{0.2\textwidth} \scriptsize{Curves of degree $e$ in $\YY$ passing though a general  point of $\YY$ and that  are  $(2e-1)$-secant to $S$  for $e\leq5$.}\end{minipage}}\tabularnewline 
     \hline \hline 

     \rowcolor{green!10}
     {\begin{tabular}{c} Quadric surface \cite{DIM} \end{tabular}}  &  $8$ & $ \sigma_{3,1}+ \sigma_{2,2}$ & $1$ & $\mathcal{GM}_{10}^{'}$ & $31$ & $8$ & $0$ & $1$, $0$, $0$, $0$, $0$\\
     \hline \rowcolor{green!5}
      \multicolumn{9}{c}{
     \begin{minipage}[t]{1.1\columnwidth}%
     \emph{M2-command}: \texttt{specialGushelMukaiFourfold "tau-quadric"} \\
     {(The same case as in the first row of Table~\ref{Table K3 associated with GMs})} 
     \end{minipage}} \\       

      \hline 
      \rowcolor{orange!10}
      {\begin{tabular}{c} Quintic Del Pezzo \\ surface \cite{Roth1949} \end{tabular}}  &  $5$ & $3 \sigma_{3,1}+2 \sigma_{2,2}$ & $1$ & $\mathcal{GM}_{10}^{''}$ & $24$ & $18$ & $3$ & $3$, $0$, $0$, $0$, $0$ \\
     \hline \rowcolor{orange!5}
      \multicolumn{9}{c}{
     \begin{minipage}[t]{1.1\columnwidth}%
     \emph{M2-command}: \texttt{specialGushelMukaiFourfold "quintic del Pezzo surface"}
     \end{minipage}} \\

     \hline
     \rowcolor{green!10}
     {\begin{tabular}{c} Cubic scroll \cite{DIM} \end{tabular}}  &  $8$ & $2 \sigma_{3,1}+\sigma_{2,2}$ & $1$ & $\mathcal{GM}_{12}$ & $28$ & $11$ & $0$ & $2$, $0$, $0$, $0$, $0$ \\
    \hline \rowcolor{green!5}
     \multicolumn{9}{c}{
    \begin{minipage}[t]{1.1\columnwidth}%
    \emph{M2-command}: \texttt{specialGushelMukaiFourfold "cubic scroll"}
    \end{minipage}} \\

    \hline
    \rowcolor{orange!10}
    {\begin{tabular}{c} Rational surface of \\ degree $10$ and genus $4$ \end{tabular}}  &  $0$ & $6 \sigma_{3,1}+4 \sigma_{2,2}$ & $1$ & $\mathcal{GM}_{16}$ & $15$ & $29$ & $5$ & $6$, $0$, $0$, $0$, $0$ \\
   \hline \rowcolor{orange!5}
    \multicolumn{9}{c}{
   \begin{minipage}[t]{1.1\columnwidth}%
   \emph{M2-command}: \texttt{specialGushelMukaiFourfold([3,5],[2,2],"quartic scroll")}
    \end{minipage}} \\    

    \hline
    \rowcolor{green!10}
    {\begin{tabular}{c} Rational surface of \\ degree $12$ and genus $5$ \end{tabular}}  &  $-1$ & $7 \sigma_{3,1}+5 \sigma_{2,2}$ & $1$ & $\mathcal{GM}_{18}^{'}$ & $11$ & $32$ & $4$ & $7$, $5$, $0$, $0$, $0$ \\
   \hline \rowcolor{green!5}
    \multicolumn{9}{c}{
   \begin{minipage}[t]{1.1\columnwidth}%
   \emph{M2-command}: \texttt{specialGushelMukaiFourfold([5,9,0,1],[2,3,0,1])}
    \end{minipage}} \\

    \hline
    \rowcolor{orange!10}
    {\begin{tabular}{c} Rational surface of \\ degree $9$ and genus $3$ \end{tabular}}  &  $2$ & $5 \sigma_{3,1}+4 \sigma_{2,2}$ & $1$ & $\mathcal{GM}_{18}^{''}$ & $16$ & $26$ & $3$ & $5$, $0$, $0$, $0$, $0$ \\
   \hline \rowcolor{orange!5}
    \multicolumn{9}{c}{
   \begin{minipage}[t]{1.1\columnwidth}%
   \emph{M2-command}: \texttt{specialGushelMukaiFourfold([4,6,1],[2,2,1])}
    \end{minipage}} \\

    \hline \rowcolor{green!10}
    {\begin{tabular}{c} Rational surface of \\ degree $9$ and genus $2$ \\ \cite{HoffSta,HoffSta2} \end{tabular}}  &  $5$ & $6 \sigma_{3,1}+3 \sigma_{2,2}$ & $1$ & $\mathcal{GM}_{20}$ & $14$ & $25$ & $0$ & $6$, $1$, $0$, $0$, $0$\\
    \hline \rowcolor{green!5}
     \multicolumn{9}{c}{
    \begin{minipage}[t]{1.1\columnwidth}%
    \emph{M2-command}: \texttt{specialGushelMukaiFourfold "general GM 4-fold of discriminant 20"} \\
    {(The same case as in the fourth row of Table~\ref{Table K3 associated with GMs})} 
    \end{minipage}} \\

    \hline
    \rowcolor{orange!10}
    {\begin{tabular}{c} Rational surface of \\ degree $10$ and genus $3$ \end{tabular}}  &  $3$ & $6 \sigma_{3,1}+4 \sigma_{2,2}$ & $1$ & $\mathcal{GM}_{24}$ & $13$ & $27$ & $1$ & $6$, $2$, $0$, $0$, $0$ \\
   \hline \rowcolor{orange!5}
    \multicolumn{9}{c}{
   \begin{minipage}[t]{1.1\columnwidth}%
   \emph{M2-command}: \texttt{specialGushelMukaiFourfold([4,5,1],[2,1,1])}
    \end{minipage}} \\    

    \hline
    \rowcolor{green!10}
    {\begin{tabular}{c} Rational surface of \\ degree $12$ and genus $5$ \end{tabular}}  &  $0$ & $7 \sigma_{3,1}+5 \sigma_{2,2}$ & $1$ & $\mathcal{GM}_{26}^{'}$ & $11$ & $30$ & $2$ & $7$, $4$, $0$, $0$, $0$ \\
   \hline \rowcolor{green!5}
    \multicolumn{9}{c}{
   \begin{minipage}[t]{1.1\columnwidth}%
   \emph{M2-command}: \texttt{specialGushelMukaiFourfold([4,9],[2,4])}
    \end{minipage}} \\    

     \end{tabular}
     %\end{adjustbox}
      \caption{\footnotesize 
      Selected codimension-one families of  GM fourfolds described as the closure of the  locus of smooth quadric hypersurfaces in a del Pezzo fivefold 
      $\YY$ 
      containing some 
      smooth
      surface $S\subset \YY$ cut out by quadrics and varying in an irreducible component of $\mathrm{Hilb}_{\YY}^{\chi(\mathcal{O}_S(t))}$. 
      See \cite[Table~1]{famGushelMukai} for more examples.}
      \label{Table uniruledness GMd d small} 
     \end{table}

\begin{table}[htbp]
     \renewcommand{\arraystretch}{1.300}
     \centering
     \tabcolsep=1.0pt 
     \footnotesize
     %\begin{adjustbox}{width=\textwidth}
     \begin{tabular}{|c|c|c|c|c|c|c|c|c|}
     \hline
     \rowcolor{gray!5.0}
       {\begin{tabular}{c} Surface $S$ in \\ $\YY=\mathbb{G}(1,4)\cap\PP^8$ \end{tabular}} & $K_S^2$ & {\begin{tabular}{c} Class in \\ $\mathbb{G}(1,4)$\end{tabular}} & {\begin{tabular}{c} Codim \\ in $\mathcal{GM}$ \end{tabular}} & \begin{tabular}{c} Locus \\ in  $\mathcal{GM}$ \end{tabular} & $h^0(\mathcal I_{S/\YY}(2))$ & $h^0(N_{S/\YY})$ & $h^0(N_{S/X})$ &
      {\begin{minipage}[c]{0.2\textwidth} \scriptsize{Curves of degree $e$ in $\YY$ passing though a general  point of $\YY$ and that  are  $(2e-1)$-secant to $S$  for $e\leq5$.}\end{minipage}}\tabularnewline 
     \hline \hline 
     
      \rowcolor{orange!10}
      {\begin{tabular}{c} Rational surface of \\ degree $13$ and genus $6$ \end{tabular}}  &  $-1$ & $8 \sigma_{3,1}+5 \sigma_{2,2}$ & $1$ & $\mathcal{GM}_{28}$ & $10$ & $31$ & $2$ & $8$, $6$, $0$, $0$, $0$ \\
     \hline \rowcolor{orange!5}
      \multicolumn{9}{c}{
     \begin{minipage}[t]{1.1\columnwidth}%
     \emph{Contruction of an example}: take the isomorphic image via $\psi_{\Sigma_3}:\PP^6\dashrightarrow \mathbb{G}(1,4)$ of a rational surface $T\subset\PP^5\subset\PP^6$ of degree $8$ and genus $4$,  obtained as the image of the plane 
     via the linear system of sextic curves with four general simple base points and six general double points, which cuts 
     $\Sigma_3$  along  a quintic elliptic curve obtained as the image of a general cubic passing through 
     three of the four simple points and five of the six double points. \\
     \emph{M2-command}: \texttt{specialGushelMukaiFourfold([6,4,6],[3,3,5])}
     \end{minipage}} \\
     
     \hline \rowcolor{green!10}
     {\begin{tabular}{c} Rational surface of \\ degree $11$ and genus $4$ \end{tabular}}  &  $2$ & $6 \sigma_{3,1}+5 \sigma_{2,2}$ & $1$ & $\mathcal{GM}_{28}$ & $12$ & $27$ & $0$ & $6$, $2$, $0$, $0$, $0$ \\
     \hline \rowcolor{green!5}
     \multicolumn{9}{c}{
     \begin{minipage}[t]{1.1\columnwidth}%
     \emph{Contruction of an example}: take the image via $\psi_{\Sigma_3}:\PP^6\dashrightarrow \mathbb{G}(1,4)$ of a rational surface $T\subset\PP^5\subset\PP^6$ of degree $7$ and genus $3$,  obtained as the  image of the plane 
     via the linear system of septic curves with six general double base points and two triple points, which cuts 
     $\Sigma_3$  along  a rational  normal quartic curve obtained as the image of a 
     conic passing through 
     five of the six double points. \\
     \emph{M2-command}: \texttt{specialGushelMukaiFourfold([7,0,6,2],[2,0,5,0])}
     \end{minipage}}
     \end{tabular}
     %\end{adjustbox}
      \caption{\footnotesize 
      Continuation of Table~\ref{Table uniruledness GMd d small}.
      New families of  GM fourfolds.}
      \label{Table uniruledness GM28} 
     \end{table}
   
   \begin{table}[htbp]
     \renewcommand{\arraystretch}{1.300}
     \centering
     \tabcolsep=1.0pt 
     \footnotesize
     %\begin{adjustbox}{width=\textwidth}
     \begin{tabular}{|c|c|c|c|c|c|c|c|c|}
     \hline
     \rowcolor{gray!5.0}
       {\begin{tabular}{c} Surface $S$ \\ in $\PP^5$ \end{tabular}} & $K_S^2$ & {\begin{tabular}{c} Nodes \end{tabular}} & {\begin{tabular}{c} Codim \\ in $\mathcal{C}$ \end{tabular}} & \begin{tabular}{c} Locus \\ in  $\mathcal{C}$ \end{tabular} & $h^0(\mathcal I_{S/\PP^5}(3))$ & $h^0(N_{S/\PP^5})$ & $h^0(N_{S/X})$ &
      {\begin{minipage}[c]{0.2\textwidth} \scriptsize{Curves of degree $e$ passing though a general  point of $\PP^5$ and that  are  $(3e-1)$-secant to $S$  for $e\leq5$.}\end{minipage}}\tabularnewline
        \hline \hline       

        \rowcolor{green!10}
        {\begin{tabular}{c} Del Pezzo surface \\ of degree $5$ \\ \cite{Fano,BRS} \end{tabular}}  &  $5$ & $0$ & $1$ & $\mathcal{C}_{14}$ & $25$ & $35$ & $5$ & $1$, $0$, $0$, $0$, $0$\\
        \hline \rowcolor{green!5}
         \multicolumn{9}{c}{
        \begin{minipage}[t]{1.1\columnwidth}%
       \emph{M2-command (1)}: \texttt{specialCubicFourfold surface \{3,4\}} \\
        \emph{M2-command (2)}: \texttt{specialCubicFourfold "quintic del Pezzo surface"}
      \end{minipage}} \\
     
      \hline \rowcolor{orange!10}
      {\begin{tabular}{c} Rational scroll \\ of degree $4$ \\ \cite{Fano,BRS} \end{tabular}}  &  $8$ & $0$ & $1$ & $\mathcal{C}_{14}$ & $28$ & $29$ & $2$ & $1$, $0$, $0$, $0$, $0$\\
      \hline \rowcolor{orange!5}
       \multicolumn{9}{c}{
      \begin{minipage}[t]{1.1\columnwidth}%
     \emph{M2-command}: \texttt{specialCubicFourfold(PP[2,2])}
    \end{minipage}} \\
   
        \hline \rowcolor{green!10}
        {\begin{tabular}{c} Del Pezzo surface \\ of degree $7$ \\ \cite{RS1} \end{tabular}}  &  $7$ & $1$ & $1$ & $\mathcal{C}_{26}$ & $14$ & $42$ & $1$ & $5$, $1$, $0$, $0$, $0$\\
        \hline \rowcolor{green!5}
         \multicolumn{9}{c}{
        \begin{minipage}[t]{1.1\columnwidth}%
        \emph{M2-command}: \texttt{specialCubicFourfold "one-nodal septic del Pezzo surface"}
      \end{minipage}} \\
        
        \hline \rowcolor{orange!10}
        {\begin{tabular}{c} Rational scroll \\ of degree $7$ \\ \cite{FV18,RS1} \end{tabular}}  &  $8$ & $3$ & $1$ & $\mathcal{C}_{26}$ & $13$ & $44$ & $2$ & $7$, $1$, $0$, $0$, $0$\\
        \hline \rowcolor{orange!5}
         \multicolumn{9}{c}{
        \begin{minipage}[t]{1.1\columnwidth}%
        \emph{M2-command}: \texttt{specialCubicFourfold "3-nodal septic scroll"}
      \end{minipage}} \\
      
     \hline \rowcolor{green!10}
     {\begin{tabular}{c} Rational surface of \\ degree $10$ and genus $6$ \\ \cite{Nuer,RS3} \end{tabular}}  &  $-1$ & $0$ & $1$ & $\mathcal{C}_{38}$ & $10$ & $47$ & $2$ & $7$, $1$, $0$, $0$, $0$\\
     \hline \rowcolor{green!5}
      \multicolumn{9}{c}{
     \begin{minipage}[t]{1.1\columnwidth}%
     \emph{M2-command (1)}: \texttt{specialCubicFourfold surface \{10,0,0,10\}} \\ 
     \emph{M2-command (2)}: \texttt{specialCubicFourfold "general cubic 4-fold of discriminant 38"}
   \end{minipage}} \\

   \hline \rowcolor{orange!10}
   {\begin{tabular}{c} Rational scroll \\ of degree $8$ \\ \cite{Explicit,RS3} \end{tabular}}  &  $8$ & $6$ & $1$ & $\mathcal{C}_{38}$ & $10$ & $47$ & $2$ & $9$, $4$, $1$, $0$, $0$\\
   \hline \rowcolor{orange!5}
    \multicolumn{9}{c}{
   \begin{minipage}[t]{1.1\columnwidth}%
   \emph{M2-command}: \texttt{specialCubicFourfold "6-nodal octic scroll"}
 \end{minipage}} \\

   \hline \rowcolor{green!10}
   {\begin{tabular}{c} Rational surface of \\ degree $9$ and genus $2$ \\ \cite{RS3} \end{tabular}}  &  $5$ & $5$ & $1$ & $\mathcal{C}_{42}$ & $9$ & $48$ & $2$ & $9$, $7$, $1$, $0$, $0$\\
     \hline \rowcolor{green!5}
      \multicolumn{9}{c}{
     \begin{minipage}[t]{1.1\columnwidth}%
     \emph{M2-command}: \texttt{specialCubicFourfold "general cubic 4-fold of discriminant 42"}
   \end{minipage}}
   
   \end{tabular}
   %\end{adjustbox}
    \caption{\footnotesize  
    Selected descriptions of
    families of rational cubic fourfolds 
    for which we can explicitly calculate associated K3 surfaces. See \cite[Table~1]{RS3} for more examples.} 
    \label{Table K3 associated with cubics} 
   \end{table}
   
   \begin{table}[htbp]
     \renewcommand{\arraystretch}{1.300}
     \centering
     \tabcolsep=1.0pt 
     \footnotesize
     %\begin{adjustbox}{width=\textwidth}
     \begin{tabular}{|c|c|c|c|c|c|c|c|c|}
     \hline
     \rowcolor{gray!5.0}
       {\begin{tabular}{c} Surface $S$ in \\ $\YY=\mathbb{G}(1,4)\cap\PP^8$ \end{tabular}} & $K_S^2$ & {\begin{tabular}{c} Class in \\ $\mathbb{G}(1,4)$\end{tabular}} & {\begin{tabular}{c} Codim \\ in $\mathcal{GM}$ \end{tabular}} & \begin{tabular}{c} Locus \\ in  $\mathcal{GM}$ \end{tabular} & $h^0(\mathcal I_{S/\YY}(2))$ & $h^0(N_{S/\YY})$ & $h^0(N_{S/X})$ &
      {\begin{minipage}[c]{0.2\textwidth} \scriptsize{Curves of degree $e$ in $\YY$ passing though a general  point of $\YY$ and that  are  $(2e-1)$-secant to $S$  for $e\leq5$.}\end{minipage}}\tabularnewline 
     \hline \hline 
     \rowcolor{orange!10}
     {\begin{tabular}{c} Quadric surface \cite{DIM} \end{tabular}}  &  $8$ & $ \sigma_{3,1}+ \sigma_{2,2}$ & $1$ & $\mathcal{GM}_{10}^{'}$ & $31$ & $8$ & $0$ & $1$, $0$, $0$, $0$, $0$\\
     \hline \rowcolor{orange!5}
      \multicolumn{9}{c}{
     \begin{minipage}[t]{1.1\columnwidth}%
     \emph{M2-command}: \texttt{specialGushelMukaiFourfold "tau-quadric"}
     \end{minipage}} \\
   
     \hline \rowcolor{green!10}
     {\begin{tabular}{c} K3 surface of degree \\ $14$ and genus $8$ \cite{Explicit} \end{tabular}}  &  $0$ & $ 9 \sigma_{3,1}+ 5 \sigma_{2,2}$ & $1$ & $\mathcal{GM}_{10}^{'}$ & $10$ & $39$ & $10$ & $9$, $8$, $1$, $0$, $0$\\
     \hline \rowcolor{green!5}
      \multicolumn{9}{c}{
     \begin{minipage}[t]{1.1\columnwidth}%
     \emph{M2-command}: \texttt{specialGushelMukaiFourfold "K3 surface of genus 8 with class (9,5)"}
     \end{minipage}} \\
   
     \hline \rowcolor{orange!10}
     {\begin{tabular}{c} Plane \cite{Roth1949} \end{tabular}}  &  $9$ & $ \sigma_{3,1}$ & $2$ & $\mathcal{GM}_{10}^{''}$ & $34$ & $4$ & $0$ & $1$, $0$, $0$, $0$, $0$\\
     \hline \rowcolor{orange!5}
      \multicolumn{9}{c}{
     \begin{minipage}[t]{1.1\columnwidth}%
     \emph{M2-command (1)}: \texttt{specialGushelMukaiFourfold "sigma-plane"} \\
     \emph{M2-command (2)}: \texttt{specialGushelMukaiFourfold schubertCycle(\{3,1\},GG(1,4))}
     \end{minipage}} \\
   
     \hline \rowcolor{green!10}
     {\begin{tabular}{c} Rational surface of \\ degree $9$ and genus $2$ \\ \cite{HoffSta,HoffSta2} \end{tabular}}  &  $5$ & $6 \sigma_{3,1}+3 \sigma_{2,2}$ & $1$ & $\mathcal{GM}_{20}$ & $14$ & $25$ & $0$ & $6$, $1$, $0$, $0$, $0$\\
     \hline \rowcolor{green!5}
      \multicolumn{9}{c}{
     \begin{minipage}[t]{1.1\columnwidth}%
     \emph{M2-command}: \texttt{specialGushelMukaiFourfold "general GM 4-fold of discriminant 20"}
     \end{minipage}} \\
   
   \hline \rowcolor{orange!10}
     {\begin{tabular}{c} Rational surface of \\ degree $13$ and genus $6$ \end{tabular}}  &  $-2$ & $8 \sigma_{3,1}+5 \sigma_{2,2}$ & $1$ & $\mathcal{GM}_{20}$ & $10$ & $33$ & $4$ & $8$, $7$, $1$, $0$, $0$\\
     \hline \rowcolor{orange!5}
      \multicolumn{9}{c}{
     \begin{minipage}[t]{1.1\columnwidth}%
     \emph{Contruction of an example}: take the isomorphic image via $\psi_{\Sigma_3}:\PP^6\dashrightarrow \mathbb{G}(1,4)$ of a rational surface $T\subset\PP^5\subset\PP^6$ of degree $8$ and genus $4$,  obtained as the image of the plane 
     via the linear system of quintic curves with nine general simple base points and two general double points, which cuts 
     $\Sigma_3$  along  a quintic elliptic curve obtained as the image of a general cubic passing through 
      six of the nine simple points and the two double points. \\
     \emph{M2-command}: \texttt{specialGushelMukaiFourfold([5,9,2],[3,6,2])}
     \end{minipage}} \\   
   
     \hline \rowcolor{green!10}
     {\begin{tabular}{c} Rational normal \\ scroll of degree $7$ \\ (see also \cite{famGushelMukai}) \end{tabular}}  &  $8$ & $4 \sigma_{3,1}+3 \sigma_{2,2}$ & $3$ & $\mathcal{GM}_{20}$ & $16$ & $21$ & $0$ & $4$, $1$, $0$, $0$, $0$\\
     \hline \rowcolor{green!5}
      \multicolumn{9}{c}{
     \begin{minipage}[t]{1.1\columnwidth}%
     \emph{Contruction of an example}: take the isomorphic image via $\psi_{\Sigma_3}:\PP^6\dashrightarrow \mathbb{G}(1,4)$ 
     of a rational normal quintic scroll $T\subset\PP^6$, obtained as the image of the plane 
     via the linear system of cubic curves with one double base point, which cuts 
     $\Sigma_3$  along  a rational normal quartic curve obtained as the image of a 
     %general 
     conic passing through 
     the double point. \\
     \emph{M2-command}: \texttt{specialGushelMukaiFourfold([3,0,1],[2,0,1])}
     \end{minipage}} \\   
     
     \hline \rowcolor{orange!10}
     {\begin{tabular}{c} Triple projection of \\ K3 surface of degree \\ $26$ and  genus $14$ \end{tabular}}  &  $-1$ & $11 \sigma_{3,1}+6 \sigma_{2,2}$ & $2$ & $\mathcal{GM}_{26}^{''}$ & $7$ & $37$ & $6$ & $11$, $22$, $32$, $6$, $1$ \\
     \hline \rowcolor{orange!5}
      \multicolumn{9}{c}{
     \begin{minipage}[t]{1.1\columnwidth}%
     \emph{Contruction of an example}: Let $X\subset \YY$ be an example of GM fourfold
     containing a surface $S$ of degree $9$ and genus $2$ as in last row, and let  $\mu:\YY\dashrightarrow\PP^5$ be
     the rational map 
     defined by the linear system of cubic hypersurfaces in $\YY$ with double points along $S$. Then 
     the base locus of $\mu$ is the union of a threefold $B$ and a $\sigma_{3,1}$-plane $P$ such that $P\cap B=P\cap S$
     is a conic.
     The projection $\xi:\YY\dashrightarrow \PP^5$ of $\YY$ from $P$
     sends the surface $S$ isomorphically into a quintic del Pezzo surface $\xi(S)\subset\PP^5$.
     The inverse image $\overline{\xi^{-1}(\xi(S))}$ is isomorphic to an internal 
     projection of a Fano threefold $\mathbb{G}(1,5)\cap \PP^{9}\subset\PP^9$ and intersects $X$
     in the union of $S$ and another smooth surface of degree $17$ and genus $11$, 
     which is isomorphic to a triple projection 
     of a minimal K3 surface of degree $26$ and genus $14$ in $\PP^{14}$.\\
     \emph{M2-command}: \footnotesize{\texttt{specialGushelMukaiFourfold("triple projection of K3 surface of degree 26")}}
     \end{minipage}} \\
     
     \hline \rowcolor{green!10}
     {\begin{tabular}{c} Rational surface of \\ degree $9$ and genus $2$ \end{tabular}}  &  $5$ & $5 \sigma_{3,1}+4 \sigma_{2,2}$ & $2$ & $\mathcal{GM}_{26}^{''}$ & $14$ & $24$ & $0$ & $5$, $1$, $0$, $0$, $0$ \\
     \hline \rowcolor{green!5}
     \multicolumn{9}{c}{
     \begin{minipage}[t]{1.1\columnwidth}%
     \emph{Contruction of an example}: take the image via $\psi_{\Sigma_3}:\PP^6\dashrightarrow \mathbb{G}(1,4)$ of a rational surface $T\subset\PP^6$ of degree $7$ and genus $2$,  obtained as the  image of the plane 
     via the linear system of quartic curves with five general simple base points and one general  double point, which cuts 
     $\Sigma_3$  along  a rational  normal quintic curve obtained as the image of a general conic passing through 
     three of the five simple points. \\
     \emph{M2-command}: \texttt{specialGushelMukaiFourfold([4,5,1],[2,3,0])}
     \end{minipage}}
   \end{tabular}
   %\end{adjustbox}
    \caption{\footnotesize  
    New and selected descriptions of
    families of rational GM fourfolds 
    for which we can explicitly calculate associated K3 surfaces. In all cases, the surface $S$ is smooth and cut out by quadrics.}
    \label{Table K3 associated with GMs} 
   \end{table}
   
   \begin{table}[htbp]
     \renewcommand{\arraystretch}{1.300}
     \centering
     \tabcolsep=1.0pt 
     \footnotesize
   %\begin{adjustbox}{width=\textwidth}
     \begin{tabular}{|c|c|c|c|c|c|c|c|c|}
          \hline
          \rowcolor{gray!5.0}
            {\begin{tabular}{c} Surface $S$ in \\ $\YY=\mathbb{G}(1,4)\cap\PP^8$ \end{tabular}} & $K_S^2$ & {\begin{tabular}{c} Class in \\ $\mathbb{G}(1,4)$\end{tabular}} & {\begin{tabular}{c} Codim \\ in $\mathcal{GM}$ \end{tabular}} & \begin{tabular}{c} Locus \\ in  $\mathcal{GM}$ \end{tabular} & $h^0(\mathcal I_{S/\YY}(2))$ & $h^0(N_{S/\YY})$ & $h^0(N_{S/X})$ &
           {\begin{minipage}[c]{0.2\textwidth} \scriptsize{Curves of degree $e$ in $\YY$ passing though a general  point of $\YY$ and that  are  $(2e-1)$-secant to $S$  for $e\leq5$.}\end{minipage}}\tabularnewline 
          \hline \hline 

          \rowcolor{orange!10}
          {\begin{tabular}{c} $1$-nodal surface of \\ degree $11$ and genus $3$ \\ (see also \cite{rationalSta}) \end{tabular}}  &  $3$ & $7 \sigma_{3,1}+4 \sigma_{2,2}$ & $2$ & $\mathcal{GM}_{26}^{''}$ & $11$ & $29$ & $2$ & $7$, $4$, $1$, $0$, $0$ \\
          \hline \rowcolor{orange!5}
           \multicolumn{9}{c}{
          \begin{minipage}[t]{1.1\columnwidth}%
          \emph{Contruction of an example}: Let $X\subset \YY$ be an example of GM fourfold
          containing a surface $S$ of degree $9$ and genus $2$ as in last row of Table~\ref{Table K3 associated with GMs}, and let  $\mu:\YY\dashrightarrow\PP^5$ be
          the rational map 
          defined by the linear system of cubic hypersurfaces in $\YY$ with double points along $S$. 
          The locus $\{p\in\overline{\mu(\YY)}:\dim\overline{\mu^{-1}(p)} \geq 2\}$ consists of a twisted cubic 
          and 
          four lines. If $L$ is one of these four lines, the inverse image $\overline{\mu^{-1}(L)}\subset \YY$ 
          is a threefold cutting $X$ in the union of the surface $S$ and a surface of degree $11$ and genus $3$ with a node.\\
          \emph{M2-command}: \footnotesize{\texttt{specialGushelMukaiFourfold("surface of degree 11 and genus 3 with class (7,4)")}}
          \end{minipage}} \\               

     \hline
     \rowcolor{green!10}
      {\begin{tabular}{c} Rational surface of \\ degree $14$ and genus $7$ \end{tabular}}  &  $-1$ & $9 \sigma_{3,1}+5 \sigma_{2,2}$ & $\_\_$ & $\mathcal{GM}_{34}^{'}$ & $9$ & $\_\_$ & $\_\_$ & $9$, $8$, $1$, $0$, $0$ \\
     \hline \rowcolor{green!5}
      \multicolumn{9}{c}{
     \begin{minipage}[t]{1.1\columnwidth}%
     \emph{Contruction of an example}: take the isomorphic image via $\psi_{\Sigma_3}:\PP^6\dashrightarrow \mathbb{G}(1,4)$ of a rational surface $T\subset\PP^5\subset\PP^6$ of degree $8$ and genus $4$,  obtained as the image of the plane 
     via the linear system of sextic curves with four general simple base points and six general double points, which cuts 
     $\Sigma_3$  along  a quintic elliptic curve obtained as the image of a general cubic passing through 
     one of the four simple points and the six double points. \\
     \emph{M2-command}: \texttt{specialGushelMukaiFourfold([6,4,6],[3,1,6])}
     \end{minipage}} %\\
%
%      \hline 
%     \rowcolor{orange!10}
%      {\begin{tabular}{c} $1$-nodal surface of \\ degree $15$ and genus $7$ \end{tabular}}  &  $-1$ & $9 \sigma_{3,1}+6 \sigma_{2,2}$ & $2$ & $\mathcal{GM}_{42}^{''}$ & $7$ & $31$ & $0$ & $9$, $18$, $22$, $10$, $1$ \\
%     \hline \rowcolor{orange!5}
%      \multicolumn{9}{c}{
%     \begin{minipage}[t]{1.1\columnwidth}%
%     \emph{Contruction of an example}: take the image via $\psi_{\Sigma_3}:\PP^6\dashrightarrow \mathbb{G}(1,4)$ of a rational surface $T\subset\PP^5\subset\PP^6$ of degree $8$ and genus $4$,  obtained as the image of the plane 
%     via the linear system of sextic curves with four general simple base points and six general double points, which cuts 
%     $\Sigma_3$  along a rational normal quartic curve obtained as the image of a line passing through 
%     two of the four simple points. \\
%     \emph{M2-command}: \texttt{specialGushelMukaiFourfold([6,4,6],[1,2,0])}
%     \end{minipage}} 
%
     \end{tabular}
     %\end{adjustbox}
      \caption{\footnotesize Continuation of Table~\ref{Table K3 associated with GMs}. 
      Note, however, that in the last row 
      the surface $S\subset\YY$ is not cut out by quadrics, and moreover
      the function \texttt{associatedK3surface}  currently 
      fails (some requirements are not met).}
      \label{Table K3 associated with GMs (further examples)} 
     \end{table}    

\begin{comment}
     \begin{table}[htbp]
          \renewcommand{\arraystretch}{1.300}
          \centering
          \tabcolsep=1.0pt 
          \footnotesize
        %\begin{adjustbox}{width=\textwidth}
          \begin{tabular}{|c|c|c|}
               \hline
               \rowcolor{gray!5.0}
                 {\begin{tabular}{c} $W$ \end{tabular}} & $U$ & {\begin{tabular}{c} Degrees of the \\ exceptional curves \end{tabular}} \\
               \hline \hline 
          \end{tabular}
          %\end{adjustbox}
           \caption{\footnotesize Further information on the 7 examples of Table~\ref{Table K3 associated with cubics}. 
           Notation as in the diagrams \eqref{diagramma} and \eqref{diagramma2}.}
           \label{Table further info} 
          \end{table}    
\end{comment}     

\clearpage
     
%\bibliographystyle{amsalpha} %  {amsplain} 
%\bibliography{../Bibliography/bibliography}

\providecommand{\bysame}{\leavevmode\hbox to3em{\hrulefill}\thinspace}
\providecommand{\MR}{\relax\ifhmode\unskip\space\fi MR }
% \MRhref is called by the amsart/book/proc definition of \MR.
\providecommand{\MRhref}[2]{%
  \href{http://www.ams.org/mathscinet-getitem?mr=#1}{#2}
}
\providecommand{\href}[2]{#2}

\end{document}